\documentclass[a4paper]{amsart}
\usepackage{amssymb, amsmath, amsthm}
\usepackage{stmaryrd, esint}
\usepackage{graphicx}
\usepackage{wrapfig}
\DeclareGraphicsExtensions{.png}

\newtheorem{theorem}{Theorem}[section]
\newtheorem{lemma}[theorem]{Lemma}

\newtheorem{corollary}[theorem]{Corollary}

\theoremstyle{definition}
\newtheorem{definition}{Definition}[section]
\newtheorem{example}[definition]{Example}

\theoremstyle{remark}
\newtheorem{remark}[definition]{Remark}

\numberwithin{equation}{section}

\newcommand{\abs}[1]{\lvert#1\rvert}
\newcommand{\Abs}[1]{\left\lvert#1\right\rvert}
\newcommand{\norm}[1]{\left\lVert#1\right\rVert}
\newcommand{\tr}[1]{\bigr\vert_{\scriptscriptstyle{#1}}}
\newcommand{\A}{\mathcal{A}}

\newcommand{\F}{\mathcal{F}}
\newcommand{\G}{\mathcal{G}}
\newcommand{\h}{\mathcal{H}}

\newcommand{\R}{\mathbb{R}}
\newcommand{\Sp}{\mathcal{S}}
\newcommand{\C}{\mathbb{C}}
\newcommand{\MN}{\mathcal{N}}
\newcommand{\N}{\mathbb{N}}
\newcommand{\Z}{\mathbb{Z}}

\DeclareMathOperator{\reg}{reg}
\DeclareMathOperator{\sing}{sing}

\DeclareMathOperator{\Div}{div}
\DeclareMathOperator{\dist}{dist}
\title[Partial H\"older continuity]{Partial H\"older continuity for $Q$-valued energy minimizing maps}

\author[J.Hirsch]{Jonas Hirsch}

\begin{document}

\begin{abstract}
We consider multivalued maps between $\Omega \subset \R^N$ open ($N \ge 2$) and a smooth, compact Riemannian manifold $\MN$ locally minimizing the Dirichlet energy. An interior partial H\"older regularity result in the spirit of R.~Schoen and K.~Uhlenbeck is presented. Consequently a minimizer is H\"older continuous outside a set of Hausdorff dimension at most $N-3$. \\
F.~Almgren's original theory includes a global interior H\"older continuity result if the minimizers are valued into some $\R^m$. It cannot hold in general if the target is changed into a Riemannian manifold, since it already fails for "classical" single valued harmonic maps.
\end{abstract}

\maketitle

\section*{Introduction}
\label{sec_H:introduction}
Multivalued maps with focus on Dirichlet integral minimizing maps have been introduced by F.~Almgren in his fundamental work \cite{Almgren}. C.~De Lellis and E.~Spadaro gave a modern revision of it in \cite{Lellis}. 
He introduced them as $Q$-valued functions. $Q \in \N$, fixed, indicates the number of values the function takes, counting multiplicity. We will refer to them from now on as $Q$-valued functions. Their purpose had been the development of a proof of a regularity result on area minimizing rectifiable currents. The author recommends \cite{Lellis note} for a motivation of their definition, an overview of Almgrens program. Furthermore it compares different modern approachs to $Q$-valued functions inspired for instance by a metric analysis and surveys some recent contributions. A complete modern revision of Almgrens original theory and results can be found in \cite{Lellis}.\\
Almgrens original Theory focuses on euclidean ambient spaces. Some of his results had been extended: \cite{Lellis select}, \cite{Goblet09}, \cite{Zhu06Fre}, \cite{dePauw} consider maps into more general ambient spaces like, metric, Banach and Hilbert spaces; \cite{Mattila83}, \cite{Lellis09} analyse a wider range of energy functionals. \cite{Zhu06An}, \cite{Zhu06flow} study different geometric objects in the $Q$-valued setting, such as geometric flows.\\
As mentioned in the abstract the H\"older continuity in the interior for minimizers is an outcome of Almgren's original work. Furthermore a minimizer $u \in W^{1,2}(\Omega, \A_Q(\R^m))$ is the "superposition" of "classical", single valued harmonic functions outside a singular set $\sing(u)$, with Hausdorff dimension not exceeding $N-2$ in the following sense:\\
$y \notin \sing(u)$, $\exists U_y \subset \Omega$ open neighborhood of $y$, $u_i: U_y \to \R^m (i=1, \dotsc, Q)$ harmonic with $u(x)=\sum_{i=1}^Q \llbracket u_i(x) \rrbracket \forall x \in U_y$. Moreover $u_i(x) = u_j(x)$ or else $u_i(x)\neq u_j(x)$ for all $x \in U_y$ and $i,j \in  \{1, \dotsc, Q \}$.\\
Almgrens regularity result has beend exteded as well in several directions. \cite{Lellis} sharpens the estimate on the singular set for planar Dirichlet minimizers, \cite{Mattila83} proves H\"older continuity for planar maps minimizing quadratic semicontinuous functionals and \cite{GobletZhu08} proof the H\"older continuity for Dirichlet almost minimizers.\\

The aim of this note is to extend the theory of harmonic maps from the single valued to the multivalued equivalent i.e. $Q$-valued maps into a smooth, compact Riemannian manifold locally minimizing the the Dirichlet integral. The interior H\"older regularity for single valued minimizing harmonic maps has been known since the work of R.~Schoen and K.~Uhlenbeck, \cite{Schoen}. In this note we give an interior partial H\"older-regular result for multivalued maps minimizing locally the Dirichlet energy. Our strategy is inspired by the methods of S.~Luckhaus, see \cite{Luckhaus}.  We are able to show:
\begin{theorem}\label{theo:0.1}
There is a constant $\alpha=\alpha(N,Q)>0$ with the property that, if $\Omega \subset \R^N (N\ge 2)$ is open, $\MN \subset \R^m$ is a smooth compact $n$-dimensional Riemannian sub-manifold and $u \in W_{loc.}^{1,2}(\Omega, \A_Q(\MN))$ is locally minimizing the Dirichlet energy, then there exists $\Omega'\subset \Omega$ open, such that $u \in C^{0,\alpha}(\Omega')$ and $\Omega\setminus \Omega'$ has at most Hausdorff dimension $N-3$.
\end{theorem}
Theorem \ref{theo:0.1} has been proved in \cite{Zhu06reg} for the special case $\MN=\Sp^n$; we use here a different approach to obtain the suitable energy decay, which allows us to cover the case of a general target manifold.\\
For single valued harmonic maps one has the following sharper result: if $\Omega, \MN$ as above and $v \in W^{1,2}_{loc.}(\Omega, \MN)$ is locally Dirichlet minimizing, then $\exists \Omega'\subset \Omega$ with $dim_\h(\Omega\setminus\Omega') \le N-3$  and $v \in C^\infty(\Omega')$. The main difference is that the $C^\infty$ regularity for single valued maps is replaced by H\"older-regularity in the multivalued setting. Furthermore we want to mention that in the single valued case the result above can be sharpened when the target manifold satisfies some special structural assumptions. \\
A pressing open question in the $Q$-valued case is to give a more detailed description of the singular set in the interior of the H\"older regular set $\Omega'$ of theorem \ref{theo:0.1}: How small is the set $\sing(u) \cap \Omega'$ s.t. $u$ can be written as a "superposition" of "classical", single valued harmonic maps. One should compare it to the corresponding result of a minimizers $u$ mapping into $\A_Q(\R^m)$ mentioned above. Another possible extension is to consider maps minimizing the $p$-Dirichlet integral in the spirit of S.~Luckhaus \cite{Luckhaus}.\\

This article is organized as follows: after fixing some notation and definitions in section \ref{sec:definition_of_energy_minimizing_maps}, we extend the "classical" variational equations and monotonicity formula to the multivalued setting in section \ref{sec:the_variational_equations_and_monotonicity_formulas}. Section \ref{sec:some_classical_results_extended_to_q_valued_functions} collects some tools to derive a compactness result for minimizers in section \ref{sec:compactness_of_energy_minimizing_maps} and the interior partial H\"older-continuity result in section \ref{sec:e-Hoelder regularity lemma}. Section \ref{sec:properties_of_the_singular_set_sing_h_u_} uses the obtained to conclude the estimate on the size of the H\"older-singular-set following classical lines. The appendix contains in section \ref{sec_I:Q-valued functions} a recollection of the most basic definitions and results concerning $Q$-valued functions taking values in $\R^n$. The results are presented omitting the actual proofs. Section \ref{sec:concentration_compactness_for_q_valued_functions} presents a concentration compactness result. It is along the same lines and indeed inspired  by C.~De Lellis and E.~Spadaro's version \cite[Lemma 3.2]{Lellis Lp}. Finally section \ref{sec:ALuckhaus}, contains an intrinsic proof to the "classical" Luckhaus' lemma concerning the extension of a map Sobolev map defined on the boundary of an annulus $\partial (B_1\setminus B_{1-\lambda})$ into the interior. The concentration compactness result and the Luckhaus' lemma provide the essential tools for the proof of theorem \ref{theo:0.1}.
\section*{Acknowledgments}
A big thank you to the many people who shared their insights and recommendations with me throughout this project. you to everyone who assisted me with conversations, new ideas, and writing their own insights in various books and articles. In particular I want to express my sincere thanks to Camillo De Lellis introducing me to F. Almgren’s $Q$-valued functions and suggesting this project to me.

\tableofcontents

\section{Definition of energy minimizing maps} 
\label{sec:definition_of_energy_minimizing_maps}
Suppose $\Omega \subset \R^N$ open, $N\ge 2$ and $\MN$ is a smooth compact $n-$dimensional Riemannian manifold isometrically embedded in some $\R^m$.

We assume for now that the reader is familiar with most general definitions and results in the theory of $Q$-valued functions. mainly the notation and terminology introduced by C.~De Lellis and E.~Spadaro in \cite{Lellis}. It differs slightly from Almgren's original one. The most basic definitions and results needed are recalled in the appendix, \ref{sec_I:Q-valued functions}, omitting the proofs.\\
For now we recall that $\left(\A_Q(\R^m), \G \right)$ denotes the metric space of unordered $Q$-tuples of points in $\R^m$. For a domain $\Omega \subset \R^N$ the Sobolev space of $Q$-valued functions is $W^{1,2}(\Omega, \A_Q(\R^m))$. The related Sobolev "semi-norm" or Dirichlet energy for $u \in W^{1,2}(\Omega, \A_Q(\R^m))$ is $\int_{\Omega} \abs{Du}^2$.\\

Following this notation $\A_Q(\MN)$ denotes corresponding to $\A_Q(\R^n)$ classical metric space of unordered $Q-$tuples taking values in $\MN$ instead of the whole $\R^n$. 

\begin{definition}\label{def:1.1}
\begin{itemize}
	\item[(i)] $W^{1,2}_{loc}(\Omega, \A_Q(\MN))$ is the set of $u \in W^{1,2}_{loc}(\Omega, \A_Q(\R^m))$ s.t. $u(x)\in \A_Q(\MN)$ for a.e. $x \in \Omega$. Since $\MN$ is assumed to be compact we have $W^{1,2}_{loc}(\Omega, \A_Q(\MN)) \subset L^\infty(\Omega, \A_Q(\R^m))$.
	\item[(ii)] For any $\overline{B_R(y)} \subset \Omega$ we define the energy $E(u, B_R(y))$ as
	\begin{equation}\label{eq:1.1}
		E(u,B_R(y))= R^{2-N} \int_{B_R(y)} \abs{Du}^2.
	\end{equation}
	\item[(iii)] We call $u \in W_{loc}^{1,2}(\Omega, \A_Q(\MN))$ a \emph{local minimizer}, or just a \emph{minimizer},  if for any $\overline{B_R(y)}\subset \Omega$ and $v \in W_{loc}^{1,2}(\Omega, \A_Q(\MN))$ satisfying $u = v$ in a neighbourhood of $\partial B_R(y)$ we have
	\begin{equation*}
		E(u,B_R(y))\le E(v,B_R(y)).
	\end{equation*}
\end{itemize}
\end{definition}

The already mentioned interior H\"older continuity result, (c.p. with \cite[Theorem 0.9]{Lellis}): for local minimizers $u \in W_{loc}^{1,2}(\Omega, \A_Q(\R^n))$ taking values in some $\R^m$ is:

\begin{theorem}[interior H\"older continuity]\label{theo_I:1.103}
There is a constant $\alpha_0=\alpha_0(N,Q)>0$ with the property that if $u \in W^{1,2}(\Omega, \A_Q(\R^n))$ is (locally) Dirichlet minimizing, then $u\in C^{0,\alpha_0}(K, \A_Q(\R^n))$ for any $K\subset \Omega\subset \R^N$ compact. Indeed, $\abs{Du}$ is an element of the Morrey space $L^{2,N-2-2\alpha_0}$ with the estimate
\begin{equation}\label{eq_I:1.103}
	r^{2-N-2\alpha_0} \int_{B_r(x)} \abs{Du}^2 \le R^{2-N-2\alpha_0} \int_{B_R(x)} \abs{Du}^2 \text{ for } r\le R, B_R(x) \subset \Omega.
\end{equation}
For two-dimensional domains $\alpha_0(2,Q)=\frac{1}{Q}$ is explicit and optimal.
\end{theorem}
As mentioned it had been proven first by Almgren in \cite{Almgren} and nicely reviewed by C. De Lellis and E. Spadaro in \cite{Lellis}.\\

We want to study the regularity of energy minimizing maps taking values in the smooth compact $n-$dimensional Riemannian manifold $\MN$. For this purpose we define the regular and singular set.
\begin{definition}\label{def:1.2}
A $Q$-valued map $u\in W^{1,2}_{loc}(\Omega, \A_Q(\R^m))$ is called regular at a point $y \in \Omega$ if there exists a neighborhood $U$ of $y$ and $Q$ smooth maps $u_i: U \to \R^m$ s.t.
\begin{equation*}
	u(x)= \sum_{i=1}^Q \llbracket u_i(x) \rrbracket \text{ for a.e. } x \in U
\end{equation*}
and either $u_i(x)\neq u_j(x)$ for all $x\in U$ or $u_i \equiv u_j$.\\
We define the open set
\begin{equation}\label{eq:1.2}
	\reg u =\left\{ y \in \Omega \,:\, y \text{ is a regular point of $u$ } \right\}. 
\end{equation}
The singular set of $u$ is the relative closed set $\sing u = \Omega \setminus \reg u$.
\end{definition}

\begin{remark}\label{rem:1.3}
If $y \in \reg u$ and $u(x)= \sum_{i=1}^Q \llbracket u_i(x) \rrbracket$ on a neighborhood $U$ each $u_i: U \to \MN$ has to be a smooth energy minimizing i.e. a harmonic map in the classical sense.  
\end{remark}

For our purpose it is helpful to define a certain subset of the singular set.

\begin{definition}\label{def:1.4}
If $u\in W^{1,2}_{loc}(\Omega, \A_Q(\R^m))$ then the H\"older regular set of $u$ is
\begin{equation}\label{eq:1.3}
	\reg_H u =\{ y \in \Omega \,:\, u \text{ is H\"older continuous in a neighborhood of } y \},
\end{equation}
and the H\"older singular set is $\sing_H u =\Omega \setminus \reg_H u$.
\end{definition}

Just by definition we have
\[
	\reg u \subset \reg_H u \text{ and } \sing_H u \subset \sing u.
\]
\begin{remark}\label{rem:1.5}
	If $\MN=\R^n$ then $\sing_H u =\emptyset$ or $\reg_H u = \Omega$ for any minimizing $u \in W^{1,2}(\Omega, \A_Q(\R^n))$ as a consequence of the internal H\"older regularity result on $\R^n$, e.g. \cite[Theorem 0.9]{Lellis}. $\sing_H u$ is not empty in general, since $\sing_H u$ is already known to be non-empty in certain cases of classical single-valued energy minimizing maps. 
\end{remark}


\section{The variational equations and monotonicity formulas} 
\label{sec:the_variational_equations_and_monotonicity_formulas}

Suppose $u \in W^{1,2}_{loc}(\Omega, \A_Q(\MN))$ is a energy minimizing map and $\overline{B_R(y)} \subset \Omega$. Suppose $\{u_t\}_{t \in ]-\delta, \delta[}$ is a $C^1$ family of maps in $W^{1,2}(B_R(y),\A_Q(\MN))$ s.t. $u_s=u$ in a neighborhood of $\partial B_R(y)$ for all $t$ and $u_0=u$ then due to minimality of $u$ we must have
\begin{equation}\label{eq:2.101}
	\frac{d}{dt}\Bigr\rvert_{t=0} E(u_t,B_R(y))=0.
\end{equation}
There two natural classes of variations, inner and outer once. 
\subsubsection{inner variations} 
\label{ssub:inner_variations}
Let $\Phi_t(x)= x + t X(x) +o(t)$ be the flow generated by an arbitrary vector field $X=(X^1, \dotsc, X^N) \in C^1_c(B_R(y),\R^N)$. Since $\Phi_t(x)=x$ close to $\partial B_R(y)$  $v_t(x)= u\circ \Phi^{-1}_t(x)$ defines a $C^1$-family of competitors to $u$. Standard calculations give 

\begin{align*}
D\Phi^{-1}_t \circ \Phi_t &= \left(D\Phi_t\right)^{-1}= \textbf{1} - t DX + o(t)\\
\det \left(D\Phi_t\right) &= 1 + t \Div(X) + o(t)
\end{align*}
( $\textbf{1}$ denotes the identity map on $\R^N$) so that 
\begin{align*}
\abs{Dv_t}^2\circ \Phi_t &= \sum_{l=1}^Q \abs{Du_l D\Phi_t^{-1}\circ \Phi_t}^2 = \sum_{l=1}^Q \abs{Du_l \left(\textbf{1} - t DX + o(t) \right)}^2\\
&= \sum_{l=1}^Q \abs{Du_l}^2 - 2 t \sum_{l=1}^Q \langle Du_l : Du_l DX \rangle + o(t).
\end{align*}
Integrating we get 
\begin{align*}
\int_{B_R(y)} \abs{Dv_t}^2 &= \int_{B_R(y)} \abs{Dv_t}^2 = \int_{B_R(y)} \abs{Dv_t}^2 \circ \Phi_t \,\abs{\det D\Phi_t} \\
&= \int_{B_R(y)} \abs{Du}^2 + t \int_{B_R(y)} \abs{Du}^2 \Div(X) - 2 \sum_{l=1}^Q \langle Du_l : Du_l DX \rangle + o(t).
\end{align*}
Because of \eqref{eq:2.101} we necessarily have 
\begin{equation}\label{eq:2.102}
	0 = \int_{B_R(y)} \left(\abs{Du}^2 \delta_{ij} - 2 \sum_{l=1}^Q \langle D_iu_l : D_ju_l \rangle\right) D_iX^j.
\end{equation}

Before we consider the second class, the outer variations, it is useful to set up some terminology and recall some facts about the nearest point projection.\\

$\MN_d=\{ x\,:\, \dist(x,\MN)<d\}$ defines a tubular neighbourhood around $\MN$ for any $d>0$. Given $p \in \MN$ and a vector $X\in \R^m$, $X^\top$ denotes the orthogonal projection of $X$ onto $T_p\MN$; hence $X^\bot= X-X^\top$ is the orthogonal projection onto the normal space $(T_p\MN)^\bot$ at $p$. $A_p(X_1^\top,X_2^\top)= (D_{X_1^{\top}}X_2^{\top})^\bot$ is the second fundamental form of $\MN$ at $p$ and any vector fields $X_1,X_2 \in C^(B_\epsilon(p), \R^m)$. 

\begin{remark}\label{rem:2.101}
Since $\MN$ is assumed to be a smooth compact manifold it has a nearest point projection $\Pi$. $\Pi$ is defined on some tubular neighbourhood $\MN_d, (d>0)$. Beside being a smooth map i.e. $\Pi \in C^\infty( \MN_d ; \MN)$ it has the following properties:
\begin{itemize}
	\item[(i)] $\abs{x-\Pi(x)}=\dist(x,\MN)<\abs{x-p}$ for all $x \in \MN_d$ and $p \in \MN\setminus \{\Pi(x)\}$;\\
	\item[(ii)] $D\Pi(p) X= X^\top$ for $p \in N$ and any vector $X\in \R^m$;\\
	\item[(iii)] for $p \in \MN$ and any vectors $X_i \in \R^N (i=1,2,3)$ 
	\begin{align}
		A_p(X_1^\top,X_2^\top)&= D^2\Pi(p)(X_1^\top,X_2^\top)\label{eq:2.001}\\
		X_1 D^2\Pi(p)(X_2,X_3) &=\sum_{\sigma \in \mathcal{P}_3} X_{\sigma(1)}^\bot D^2\Pi(p)( X_{\sigma(2)}^\top, X_{\sigma(3)}^\top)\label{eq:2.002}
	\end{align}
	\item[(iv)] for any $x \in \MN_d$ and any vector  $X \in \R^m$ we have
	\[
		\left(1 - 2 \operatorname{dist}(x,\MN) \norm{A_{\Pi(x)}}\right)\;\abs{D\Pi(x)X}^2 \le \abs{X}^2.
	\]
\end{itemize}	
\end{remark}

Although all of these are classical, we give their proofs expect for showing existence and smoothness of $\Pi$, that can be found for example in \cite[2.12.3 Theorem 1]{Simon}.
\begin{itemize}
	\item[(i)] is the defining property of $\Pi$ as nearest point projection.\\
	\item[(ii)] For $X \in \R^m$ given, we may write $X=X^\top +X^\bot$. Take a curve $\gamma: ]-\delta, \delta [ \to N$ with $\gamma(0)=p$ and $\gamma'(0)=X^\top$. Since $\Pi(\gamma(t))=\gamma(t)$ we have
\[ \abs{\Pi(p+t X^\top)-\gamma(t)} \le \abs{p+t X^\top - \Pi(p + t X^\top)} + \abs{p+ t X^\top - \gamma(t)} \le 2\abs{p + tX^\top - \gamma(t)}\] that is of order $o(t)$ and so $D\Pi_pX^\top = X^\top$. Since $\Pi(p+t X^\bot)=p$ for all $t$ we conclude
	\[
		D\Pi(p)(X)= \frac{d}{dt}\Bigr\rvert_{0} \Pi(p+ t X^\top + t X^\bot) = \frac{d}{dt}\Bigr\rvert_0 \Pi(p+ t X^\top) + \frac{d}{dt}\Bigr\rvert_{0} \Pi(p + t X^\bot) = X^\top;
	\]
	\item[(iii)] Let $X_2 \in C^\infty(\MN_d, \R^m)$ and let $\gamma$ be a curve in $\MN$ with $\gamma(0)=p, \gamma'(0)=X_1^\top$. Differentiating $X_2^\top\circ \gamma(t) = D\Pi(\gamma(t)) X_2(\gamma(t))$ we deduce
	\[
		D_{X_1^\top}(X_2^\top)= (D_{X_1^\top}X_2)^\top + D^2\Pi(p)(X_1^\top,X_2)
	\]
	with the particular choices $X_2=X_2^\top, X_2=X_2^\bot$ we reach
	\begin{align*}
		D^2\Pi(p)(X_1^\top,X_2^\top)&= D_{X_1^\top}(X_2^\top) - (D_{X_1^\top}X_2^\top)^\top = (D_{X_1^\top}X_2^\top)^\bot = A_p(X_1^\top,X_2^\top),\\
		D^2\Pi(p)(X_1^\top,X_2^\bot)&= - (D_{X_1^\top}X_2^\bot)^\top.
	\end{align*}
	Recall that $\langle X_2^\top, X_3^\bot\rangle =0$ implies $0=\langle D_{X_1^\top}X_2^\top, X_3^\bot \rangle + \langle X_2^\top, D_{X_1^\top}X_3^\bot \rangle$.  Additionally one has $D^2\Pi(p)(X_2^\bot, X_3^\bot)=0$ since $p= \Pi(p)= \Pi(p + s X_2^\bot + t X_3^\bot)$ for all $s,t$. A short calculation give the desired conclusion \eqref{eq:2.002}.
	\item[(iv)] Let $\gamma:]-\delta, \delta[ \to \MN_d$ be any $C^2$ curve in the tubular neighborhood of $\MN$ with $\gamma(0)=x$ and $\gamma'(0)=X$ e.g. $\gamma(t)=x+tX$. Define $\tilde{\gamma}(t)= \Pi(\gamma(t))$ the corresponding $C^2$ curve on $\MN$. Hence $\tilde{\gamma}'(0)=D\Pi(x) X$. Set $\nu(t)=\gamma(t)-\tilde{\gamma}(t)$ i.e. $\abs{\nu(t)}=\dist(\gamma(t),\MN)$ and hence $0=\langle \tilde{\gamma}'(t), \nu(t) \rangle$. Differentiating we obtain
	\[
		\langle \tilde{\gamma}'(t), \nu'(t) \rangle = -\langle \tilde{\gamma}''(t), \nu(t) \rangle = -\langle A_{\tilde{\gamma}(t)}(\tilde{\gamma}'(t), \tilde{\gamma}'(t)), \nu(t) \rangle.
	\]
	Furthermore we find
	\begin{align*}
		\abs{X}^2&= \abs{\tilde{\gamma}'(0)}^2 + 2 	\langle \tilde{\gamma}'(0), \nu'(0) \rangle + 	\abs{\nu'(0)}^2\\
		&\ge \abs{\tilde{\gamma}'(0)}^2- 2\langle A_{\tilde{\gamma}(0)}(\tilde{\gamma}'(0), \tilde{\gamma}'(0)), \nu(0) \rangle\\
		&\ge \abs{\tilde{\gamma}'(0)}^2 - 2\dist(x,\MN) \norm{A_{\Pi(x)}} \abs{\tilde{\gamma}'(0)}^2.
	\end{align*}
\end{itemize}

Now we are ready to consider outer variations.

\subsubsection{Outer variations} 
\label{ssub:outer_variations}
Let $Y=(Y^1, \dotsc, Y^n) \in C^1(B_R(y)\times \R^n, \R^n)$ be an arbitrary vector field with $Y(x,z)=0$ for $x$ close to $\partial B_R(y)$. Set $\Psi_t(x,z)=z+t Y(x,z)$. For sufficiently small $t$ we obtain a $C^1$ family of competitors setting
\[
	v_t(x)= \Pi(\Psi_t(x,u(x))).
\]
\[
	\Pi(p+ t Y(x,p))= p + t \int_{0}^1 D_i\Pi(p+st Y(x,p)) Y^i(x,p) \, ds \quad \forall p\in \MN \text{ and small } t
\]
and apply the chain rule (\cite[Proposition 1.12]{Lellis}) for the $1-$jet of $v_t$ to conclude
\begin{align*}
	&J\mathcal{V}_{t,x}(y)=\sum_{l=1}^Q \llbracket u_l(x)+ o(t) +\\
	 &t \left( D\Pi(u_l(x))D_iY(x,u_l(x)) + D^2\Pi(u_l(x))(D_iu(x),Y(x,u_l(x)))\right)  (y^i-x^i)\rrbracket.
\end{align*}
$\langle D_iu_l, D^2\Pi(u_l)(D_iu_l, Y)\rangle = \langle Y, A_{u_l}(D_iu_l,D_iu_l)\rangle$ as seen in remark \ref{rem:2.101} (iii), since $D_iu_l(x) \in T_{u_l(x)}\MN$ for a.e. $x$. \eqref{eq:2.101} necessarily implies that
\begin{equation}\label{eq:2.103}
	0= \int_{B_R(y)} \left(\sum_{i=1}^N \sum_{l=1}^Q  \langle D_iu_l, D_iY(x,u_l)\rangle + \langle A_{u_l}(D_iu_l,D_iu_l), Y(x,u_l)\rangle\right).
\end{equation}

\subsubsection{Monotonicity formulas:} 
\label{ssub:monotonicity_formulas_}
Let $u \in W^{1,2}_{loc}(\Omega, \A_Q(\MN))$ be an energy minimizing map and $\overline{B_R(y)} \subset \Omega$. For a.e. $0<r\le R$ we have
\begin{equation}\label{eq:2.104}
	\int_{B_r(y)} \left(\sum_{i=1}^N \sum_{l=1}^Q  \langle D_iu_l, D_iu_l\rangle + \langle A_{u_l}(D_iu_l,D_iu_l), u_l\rangle\right) = \int_{\partial B_r(y)} \sum_{l=1}^Q \langle u_l, \frac{\partial u_l}{\partial r}\rangle,
\end{equation}
and 
\begin{equation}\label{eq:2.105}
	(2-N) \int_{B_r(y)} \abs{Du}^2 = 2r \int_{\partial B_r(y)} \Abs{\frac{\partial u}{\partial r}}^2 - r \int_{\partial B_r(y)} \abs{Du}^2.
\end{equation}

To conclude these two identities recall the following general fact from analysis: if $a=(a^1, \ldots , a^N) \in L^1(B_R(y), \R^N)$, $f \in L^1(B_R(y), \R)$ satisfies \[\int_{B_R(y)} a^i D_i \varphi = \int_{B_R(y)} f \varphi \quad \forall \varphi \in C^{\infty}_c(B_R(y))\] then
\begin{equation}\label{eq:2.106}
	\int_{B_r(y)} a^iD_i\phi - f\phi = \int_{\partial B_r(y)} \phi \langle a, \nu \rangle \quad \forall \phi \in C^1(\overline{B_R(y)}), \text{ a.e.} 0<r<R.
\end{equation}
(This may be checked approximating the function $\textbf{1}_{B_r(y)}$ with smooth functions.)

To deduce \eqref{eq:2.104} choose the vector field $Y(x,z)=\varphi(x) z, \varphi \in C^{\infty}_c(B_R(y), \R)$ in the outer variation, hence $0= \int_{B_R(y)} a^iD_i\varphi - f \varphi$ with $a^i=\sum_{l=1}^Q \langle D_iu_l, u_l \rangle$ and $-f = \left(\sum_{l=1}^Q \abs{Du_l}^2 + \sum_{j=1}^N \langle A_{u_l}(D_ju_l, D_ju_l), u_l \rangle \right)$. Hence \eqref{eq:2.104} follows from \eqref{eq:2.106} with $\phi=1$.\\
\eqref{eq:2.105} can be checked similarly. Apply \eqref{eq:2.106} for every $j$ separately with the choice $\phi^j(x)=(x^j-y^j)$ ($D_i\phi^j= \delta_{ij}$). Take then sum in $j$ and conclude \eqref{eq:2.105}.\\
\eqref{eq:2.105} can be considered as a differential identity. If one fix some $0<s<R$, then \eqref{eq:2.105} implies that for a.e. $s\le r \le R$
\begin{align*}
	\frac{d}{dr} r^{2-N} \int_{B_r(y)} \abs{Du}^2 &= r^{1-N} (2-N) \int_{B_r(y)} \abs{Du}^2 + r^{2-N} \int_{\partial B_r(y)} \abs{Du}^2 \\
	&=2 r^{2-N} \int_{\partial B_r(y)} \Abs{\frac{\partial u}{\partial r}}^2= 2 \frac{d}{dr} \int_{B_r(y)\setminus B_r(y)} \abs{x-y}^{2-N} \Abs{\frac{\partial u}{\partial r}}^2.
\end{align*}
$r \mapsto \int_{B_r(y)} f$ is an absolutely continuous function for any $f \in L^1$. So we can integrate the differential identity above and conclude the classical \emph{monotonicity formula} for $0<s\le r\le R$:\\
\begin{equation}\label{eq:2.107}
	r^{2-N} \int_{B_r(y)} \abs{Du}^2 - s^{2-N} \int_{B_s(y)} \abs{Du}^2=2 \int_{B_r(y)\setminus B_s(y)} \abs{x-y}^{2-N} \Abs{\frac{\partial u}{\partial r}}^2.
\end{equation}

Notice that, due to the positivity of the right side in \eqref{eq:2.107} $r \mapsto E(u,B_r(y))$, is non decreasing and its limit exists. 

\begin{definition}\label{def:2.102}
	We define the density function $\Theta_u$ of $u$ on $\Omega$ by
	\begin{equation}\label{eq:2.108}
		\Theta_u(y)=\lim_{r \to 0} E(u,B_r(y)).
	\end{equation}
\end{definition}
Just note that \eqref{eq:2.107} reduces in the limit $s \to 0$ to
\begin{equation}\label{eq:2.109}
	E(u,B_r(y)) - \Theta_u(y) = 2 \int_{B_r(y)} \abs{x-y}^{2-N} \Abs{\frac{\partial u}{\partial r}}^2.
\end{equation}

\section{The Luckhaus lemma extended to $Q$-valued functions} 
\label{sec:some_classical_results_extended_to_q_valued_functions}
In this section we recall a result of S.~Luckhaus, \cite[Lemma 1]{Luckhaus} and extend it to $Q$-valued functions. As for single valued maps it is an essential tool in the proof of theorem \ref{theo:0.1}. We state it in a formulation due to R.~Moser in \cite{Moser}.

\begin{lemma}\label{lem:3.101}
There is a constant $C=C(N,m,Q)$ such that: given $0<\lambda<\frac{1}{2}$ and $u,v \in W^{1,2}(\Sp^{N-1}, \A_Q(\R^m))$ with 
\[ \int_{\Sp^{N-1}} \abs{D_\tau u }^2 + \abs{D_\tau v}^2 + \frac{\G(u,v)^2}{\epsilon^2} =K^2 \]
for some $0<\epsilon<\lambda$, then there exists $\varphi \in W^{1,2}(B_1\setminus B_{1-\lambda}, \A_Q{\R^m})$ with the following properties
\begin{align}\label{eq:L001}
&\varphi(x)= \begin{cases} u(x), &\text{ if } \abs{x}=1\\ v(\frac{x}{1-\lambda}), &\text{ if } \abs{x}=1-\lambda \end{cases} \\ \label{eq:L002}
&\int_{B_1\setminus B_{1-\lambda}} \abs{D\varphi}^2 \le C \lambda \left( \int_{\Sp^{N-1}} \abs{D_\tau u }^2 + \abs{D_\tau v}^2 + \frac{\G(u,v)^2}{\lambda^2} \right)\le C\, \lambda K^2\\
&\varphi(x) \in \{ y \in \R^m \colon \dist(y, u(\Sp^{N-1}) \cup v(\Sp^{N-1}) )< a\}  \text{ for some $a>0$ with } \label{eq:L003}\\ \nonumber
 &a^2\le \frac{C}{\lambda^{N-2}} \left(\int_{\Sp^{N-1}} \abs{D_\tau u }^2 + \abs{D_\tau v}^2\right)^{\frac{1}{2}} \left(\int_{\Sp^{N-1}} \G(u,v)^2\right)^{\frac{1}{2}} + \frac{C}{\lambda^{N-1}} \int_{\Sp^{N-1}} \G(u,v)^2
 \\ \nonumber &\le C_\infty \, Q^2\, \lambda^{2-N} \epsilon K^2 \nonumber.
\end{align}
%
\end{lemma}

\begin{proof}
The lemma can be concluded directly from Moser's argument, see \cite[Lemma 4.4]{Moser} using Almgren's bilipschitz embedding $\boldsymbol{\xi}$. \\
Before we deduce it from Moser's result for $Q$-valued function, we shortly describe how to get the estimates with $K$ from the Moser's result. The first is just \[ \int_{\Sp^{N-1}} \abs{D_\tau u }^2 + \abs{D_\tau v}^2 + \frac{\G(u,v)^2}{\lambda^2} \le \int_{\Sp^{N-1}} \abs{D_\tau u }^2 + \abs{D_\tau v}^2 + \frac{\G(u,v)^2}{\epsilon^2}. \]
The second follows by Cauchy's inequality:
\begin{align*}
&2\epsilon \left(\int_{\Sp^{N-1}} \abs{D_\tau u }^2 + \abs{D_\tau v}^2\right)^{\frac{1}{2}} \left(\int_{\Sp^{N-1}} \frac{\G(u,v)^2}{\epsilon^2}\right)^{\frac{1}{2}} \\
&\le \epsilon \left(\int_{\Sp^{N-1}} \abs{D_\tau u }^2 + \abs{D_\tau v}^2 + \int_{\Sp^{N-1}} \G(u,v)^2\right).
\end{align*}
To derive the result for $Q$-valued functions one can argue as follows: Given $u,v$ as stated, $\boldsymbol{\xi}\circ u, \boldsymbol{\xi}\circ v \in W^{1,2}(\Sp^{N-1}, \R^{\tilde{m}})$ are admissible, and all integral quantities are comparable up to a constant $C(N,m,Q)>0$. By \cite[Lemma 4.4]{Moser} there exists a single valued function  $\tilde{\varphi} \in W^{1,2}(B_1\setminus B_{1-\lambda}, \R^{\tilde{m}})$ that has the stated properties, replacing $u$ by $\boldsymbol{\xi}\circ u$ and $v$ by $\boldsymbol{\xi}\circ v$. Set $\varphi= \boldsymbol{\rho}\circ \tilde{\varphi} \in W^{1,2}(B_1\setminus B_{1-\lambda}, \A_Q(\R^m))$ using Almgren's retraction $\boldsymbol{\rho}$. $\varphi$ then has the desired properties, since again all integral quantities are comparable up to a constant $C(N,m,Q)>0$.
\end{proof}
For the sake of completeness we presented here the "simple" argument based on Almgren's bilipschitz embedding $\boldsymbol{\xi}$. The appendix \ref{sec:ALuckhaus} discusses this result in more detail and contains an intrinsic proof.

\section{Compactness of energy minimizing maps} 
\label{sec:compactness_of_energy_minimizing_maps}
We can follow the classical argument to get as a consequence of lemma \ref{lem:3.101} a compactness result for energy minimizing maps (compare \cite[section 2.9 Lemma 1]{Simon}).
\begin{lemma}\label{lem:4.1}
If $\{u(k)\}\subset W^{1,2}(\Omega, \A_Q(\MN))$ is a sequence of energy minimizing with $\limsup_{k\to \infty} E(u(k), B_\rho(y)) <\infty$ for each ball $\overline{B_R(y)}\subset \Omega$, then there is a subsequence $\{u(k')\}$ and a energy minimizer $u \in W^{1,2}(\Omega, \A_Q(\MN))$ s.t. 
\begin{itemize}
	\item[(i)] $\G(u(k'),u) \to 0$ in $L^2(\Omega)$\\
	\item[(ii)] $\lim_{k'\to \infty} E(u(k'), B_R(y))= E(u,B_R(y))$ for every ball $\overline{B_R(y)}\subset \Omega$.
\end{itemize}
\end{lemma}

\begin{proof}
$\norm{u(k)}_{L^\infty(\Omega)}<C$ for all $k$ because $\MN$ is compact by assumption. So that 
\[
	\limsup_{k \to \infty} \int_{\Omega'} \abs{u(k)}^2 + \abs{Du(k)}^2 < \infty
\]
for every $\Omega' \subset\subset \Omega$. By Rellich's compactness theorem for bounded sequences in $W^{1,2}$ there is a subsequence not relabeled $u(k)$ and a $u \in W^{1,2}(\Omega, \A_Q(\R^m))$ s.t. $\G(u(k),u)(x) \to 0$ in $L^2$ and a.e. $x \in \Omega$. Hence $u(x)\in \A_Q(\MN)$ for a.e. $x \in \Omega$. It remains to prove that
\[
	\lim_{k \to \infty } E(u(k), B_R(y)) = E(u,B_R(y)) \quad\forall \overline{B_R(y)} \subset \Omega.
\]
Let be $B_R(y)$ be given, not changing notation we write $u(k)(x)$ for $u(k)(y+Rx)$, so we can assume that $B_R(y)$ is the unite ball $B_1$. So $\G(u(k), u) \to 0$ in $L^1(B_1)$ and there is $K>0$ with $\liminf_{k \to \infty} \int_{B_1} \abs{Du(k)}^2 \le K^2$.\\
Let $\frac{1}{2} < r_1 <1$ and $0< \delta< 1-r_1$ arbitray small but fixed. Then fix $0<\epsilon < \lambda < \frac{\delta}{3}$ s.t. if $C$ is the constant of \ref{lem:3.101}, $d$ the size of the tubular neighborhood, then
\[ C \frac{2^{N+2}}{1-r_1} K^2 \lambda < \delta \text{ and } C \frac{2^{N+3}}{\lambda^{N-2} (1-r_1)} \epsilon < d^2 \]
For $u(k)_r(x)= u(k)(rx)$, $u_r(x)=u(rx)$, Fatou's lemma states
\begin{align*}
&\int_{r_1}^1 \liminf_{k \to \infty}  \int_{\Sp^{N-1}} \abs{Du(k)_r}^2 + \abs{Du_r}^2 + \frac{\G(u(k)_r, u_r)^2}{\epsilon^2} \\
&\le \liminf_{k \to \infty} \int_{r_1}^1 r^{-N} \int_{\partial B_r} r^2 (\abs{Du(k)}^2 + \abs{Du}^2) + \frac{\G(u(k),u)^2}{\epsilon^2} \\
&\le 2^{N} \liminf_{k \to \infty} \int_{B_1\setminus B_{1-\lambda}} \abs{Du(k)}^2 + \abs{Du}^2 + \frac{\G(u(k),u)^2}{\epsilon^2} \le 2^{N+1} K^2
\end{align*}
Hence there is a radius $\rho$, $0<r_1<\rho<1$ and a subsequence $u(k)$ not relabelled with 
\[ \int_{\Sp^{N-1}} \abs{Du(k)_\rho}^2 + \abs{Du_\rho}^2 + \frac{\G(u(k)_\rho, u_\rho)^2}{\epsilon^2} < \frac{2^{N+2}}{1-r_1} K^2. \]
We apply the Luckhaus' lemma \ref{lem:3.101} to each tuple $u(k)_\rho, u_\rho$ and obtain $\tilde{\varphi}(k) \in W^{1,2}(B_1\setminus B_{1-\lambda}, \A_Q(\R^m))$ with $\tilde{\varphi}(k)(x)= u(k)_\rho(x)$ for $\abs{x}=1$, $\tilde{\varphi}(k)(x)= u_\rho(\frac{x}{1-\lambda})$ for $\abs{x}=1-\lambda$, $\int_{B_1\setminus B_{1-\lambda}} \abs{D\tilde{\varphi}(k)}^2 \le   C \frac{2^{N+2}}{1-r_1} K^2 \lambda < \delta$ and
\[ \tilde{\varphi}(k)(x) \in \{ z \colon \dist(z, u(k)(\partial B_\rho) \cup u(\partial B_\rho))< a\} \]
with  $a^2 \le C \frac{2^{N+3}}{\lambda^{N-2} (1-r_1)} \epsilon < d^2 $.\\
Therefore $\varphi(k)(x)=\Pi(\tilde{\varphi}(k)(\frac{x}{\rho}))$ is well defined and satisfies $\varphi(k)(x)= u(k)(x)$ for $\abs{x}=\rho$, $\varphi(k)(x)= u(\frac{x}{1-\lambda})$ for $\abs{x}=(1-\lambda)\rho$ and 
\[ \int_{B_\rho\setminus B_{(1-\lambda)\rho}} \abs{D\varphi(k)}^2 \le C \rho^{N-2} \int_{B_\rho\setminus B_{(1-\lambda)\rho}} \abs{D\tilde{\varphi}(k)}^2(\frac{x}{\rho}) \frac{dx}{\rho^N} \le C\delta \rho^{N-2}. \]  
Given a competitor $v \in W^{1,2}(B_\rho, \A_Q(\MN))$ to $u$, the map
\[ v(k)=\begin{cases}
u(k), &\text{ for } \rho \le \abs{x} \le 1 \\
\varphi(k), &\text{ for } (1-\lambda)\rho \le \abs{x} \le \rho \\
v(\frac{x}{1-\lambda}), \text{ for }  \abs{x} \le (1-\lambda)\rho \\
\end{cases}\]
defines a competitor to $u(k)$. Hence by minimality of each $u(k)$ we got
\[ \int_{B_1} \abs{Du(k)}^2 \le \int_{B_1} \abs{Dv(k)}^2 \le \int_{B_1\setminus B_\rho} + C \rho^{N-2} \delta + (1- \lambda)^{N-2} \int_{B_\rho} \abs{Dv}^2 \]
or
\[ \int_{B_\rho} \abs{Du(k)}^2 \le C\delta + \int_{B_\rho} \abs{Du}^2. \]
This implies that
\[ \int_{B_\rho} \abs{Du}^2 \le \liminf_{k \to \infty} \int_{B_\rho} \abs{Du(k)}^2 \le C \delta + \int_{B_\rho} \abs{Dv}^2.\]
$\delta>0$ had been arbitrary small, so $u$ is minimizing on each $B_{r_1} \subset B_\rho \subset B_1$. Choose $u=v$ to deduce the strong convergence of energy, (ii).
\end{proof}


\section{$\epsilon$-H\"older regularity lemma}
\label{sec:e-Hoelder regularity lemma}
In this section we are going to prove an $\epsilon$-regularity lemma for the H\"older continuity of energy minimizing maps in the spirit of the Schoen-Uhlenbeck regularity theorem. 

\begin{lemma}\label{lem:5.101}
Let $u(k) \in W^{1,2}(B_1, \A_Q(\MN)$ be a sequence of energy minimizers with
\[ \lim_{k \to \infty} E(u(k), B_1) =0. \]
For a subsequence, not relabled we can find $a_l \in W^{1,2}(B_1, \A_{Q_l}(\R^m))$, $\sum_{l=1}^L Q_l =Q$, a sequence of points $p_l(k) \in \MN$ s.t. 
\begin{itemize}
\item[(i)] $a_l$ is Dirichlet minimizing, \\
\item[(ii)] $\G(\sigma_k^{-1} u(k), a(k)) \to 0$ in $L^1(B_1)$ for $\sigma_k^2 =E(u(k), B_1)$, $a(k)=\sum_{l=1}^L a_l \oplus \sigma_k^{-1} p_l(k)$\\
\item[(iii)] $\lim_{k \to \infty} \sigma_k^{-2} E(u(k), B_R) = \sum_{l=1}^L \int_{B_R} \abs{Da_l}^2 \text{ for any } 0<R<1.$
\end{itemize}
\end{lemma}

\begin{proof}
To every $u(k)$ fix a "mean" $T(k) \in \A_Q(\R^m)$. We may assume that $T(k) \in \A_Q(\MN)$, if not replace it by $T \in \A_Q(\MN)$ with $\G(T(k),T)= \min_{S \in A_Q(\MN)} \G(T(k),S)$. $T$ has still the property of a "mean", as one may check:
\begin{align*}
\int_{B_1} \G(u(k),T)^2 &\le 2 \int_{B_1} \G(u(k),T(k))^2 + 2 \int_{B_1} \G(T(k),T)^2\\ &\le 4 \int_{B_1} \G(u(k),T(k))^2 \le 4C \int_{B_1} \abs{Du(k)}^2.
\end{align*}
Next we apply the concentration compactness lemma, \ref{lem_I:A1.1}, to the sequence of tuples $\sigma_k^{-1} u(k), \sigma_k^{-1} T(k)$. For a subsequence not relabeled, we get maps $a_l \in W^{1,2}(\A_{Q_l}(B_1),\R^m)$, a related sequence $t_l(k)=\sigma_k^{-1}p_l(k) \in spt(\sigma_k^{-1} T(k))$. (ii) is a consequence of the concentration compactness lemma. It remains to prove that the $a_l$'s are Dirichlet minimizing and that the strong convergence of energy (iii) holds. The concentration compactness lemma implies $\sum_{l=1}^L \int_{B_R} \abs{Da_l}^2 \le \liminf_{k \to \infty} \sigma_k^{-1} E(u(k),B_R)$ for all $0<R<1$ as a consequence of the lower semicontinuity of energy.\\
$\psi_k$ denotes the $1-$Lipschitz retraction map onto $B_{\sigma_k^{-\frac{1}{2}}}\subset \R^m$ defined as 
\[\psi_k(z)=\begin{cases}
z, &\text{ for } \abs{z} < \sigma_k^{-\frac{1}{2}}\\
\sigma_k^{-\frac{1}{2}} \frac{z}{\abs{z}}, \text{ for } \abs{z} \ge \sigma_k^{-\frac{1}{2}}
\end{cases}\]
Furthermore we set
\begin{equation}\label{eq:5.101}
a(k)= \sum_{l=1}^L a_l \oplus \sigma_k^{-1} p_l(k) \text{ and } \tilde{a}(k)=\sum_{l=1}^L \psi_k(a_l) \oplus \sigma_k^{-1} p_l(k).
\end{equation}
We still have $\G(\tilde{a}(k), \sigma_k^{-1} u(k))\to 0$ in $L^2(B_1)$ because
\begin{align*}
\int_{B_1} \G(\tilde{a}(k), \sigma_k^{-1}u(k))^2 &\le 2 \int_{B_1} \G(a(k), \sigma_k^{-1}u(k))^2 + 2\int_{B_1} \G(a(k),\tilde{a}(k))^2\\
&\le 2 \int_{B_1} \G(a(k), \sigma_k^{-1}u(k))^2 + 2 \sum_{l=1}^L \int_{B_1 \cap \{ x\colon \abs{a_l}(x) \ge \sigma_k^{-\frac{1}{2}}\}} \abs{a_l}^2.
\end{align*}
$\int_{B_1} \G(a(k), \sigma_k^{-1}u(k))^2 \to 0$ in $L^2(B_1)$ due to the concentration compactness lemma, the second integral tends to $0$ as $k \to \infty$ (since $\sigma_k \to 0$) and $\abs{a_l} \in L^2(B_1)$.\\
Let $\frac{1}{2}<R<1$ be fixed and $0<\delta <1-R$ be arbitrary small but fixed as well. Choose $0<\epsilon<\lambda<\frac{\delta}{3}$ s.t. 
\[ C\frac{2^{N+2}}{1-R} \lambda < \delta \quad \text{ and }\quad C \frac{2^{N+3}}{\lambda^{N-2}(1-R)} \epsilon < \frac{1}{2} d^2, \]
where $C$ is the constant of the Luckhaus' lemma \ref{lem:3.101} and $d>0$ the size of the tubular neighbourhood to $\MN$.\\
Using the classical notation for rescalling a function $f_r(x)=f(rx)$, Fatou's lemma states 
\begin{align*}
&\int_{1-R}^1 \liminf_{k\to \infty} \int_{\Sp^{N-1}} \left(\frac{\abs{Du(k)_r}^2}{\sigma_k^2} + \abs{Da(k)_r}^2 + \frac{\G(\tilde{a}(k)_r,\sigma_k^{-1} u(k)_r)^2}{\epsilon^2}\right)\\
&\le \liminf_{k\to \infty}\int_{1-R}^1 r^{-N} \int_{\partial B_r} \left( r^2\left(\frac{\abs{Du(k)_r}^2}{\sigma_k^2} + \abs{Da(k)_r}^2\right) + \frac{\G(\tilde{a}(k)_r,\sigma_k^{-1} u(k)_r)^2}{\epsilon^2}\right) \\
& \le 2^N \liminf_{k\to \infty} \int_{B_1\setminus B_{1-R}} \left( \frac{\abs{Du(k)_r}^2}{\sigma_k^2} + \abs{Da(k)_r}^2 + \frac{\G(\tilde{a}(k)_r,\sigma_k^{-1} u(k)_r)^2}{\epsilon^2}\right) \le 2^{N-1}.
\end{align*}
Hence there must be a radius $R<\rho<1$ and a subsequence not relabelled with
\[ \int_{\Sp^{N-1}} \left( \frac{\abs{Du(k)_\rho}^2}{\sigma_k^2} + \abs{Da(k)_\rho}^2 + \frac{\G(\tilde{a}(k)_\rho,\sigma_k^{-1} u(k)_\rho)^2}{\epsilon^2} \right)< \frac{2^{N+2}}{1-R}.\]
As in the proof of the compactness of minimizers, lemma \ref{lem:4.1}, apply Luckhaus' lemma \ref{lem:3.101} to each tuple $\sigma_k^{-1} u(k)_\rho, \tilde{a}(k)_\rho$ to get $\tilde{\varphi}(k) \in W^{1,2}(B_1\setminus B_{1-\lambda}, \A_Q(\R^m))$ with the properties that
\begin{itemize}
\item[(i)] $\tilde{\varphi}(k)(x)= \sigma_k^{-1} u(k)_\rho(x)$  for  $\abs{x}=1$ and $\tilde{\varphi}(k)(x)=\tilde{a}(k)_\rho(\frac{x}{1-\lambda})$ for $ \abs{x}=1-\lambda$;\\
\item[(ii)] $\int_{B_1\setminus B_{1-\lambda}} \abs{D\tilde{\varphi}(k)}^2 \le C\frac{2^{N+2}}{1-R}\lambda < \delta$;\\
\item[(iii)] for some $a^2 \le C \frac{2^{N+3}}{\lambda^{N-2}(1-R)} \epsilon < \frac{1}{2} d^2$ \[ \tilde{\varphi}(k)(x) \in  \{ z \in \R^m \colon \dist(z, \sigma^{-1}_k u(k)(\partial B_\rho) \cup \tilde{a}(k)(\partial B_\rho) )< a \}. \]
\end{itemize}
Due to \eqref{eq:5.101} we have $\dist(\sigma_k^{-1} \tilde{a}(k)(x), \A_Q(\MN))^2 \le \sum_{l=1}^L \sigma_k^2 \abs{\psi_k(a(k))(x)}^2 \le L \sigma_k$, so that $\dist(\sigma_k \tilde{\varphi}(k)(x), \A_Q(\MN))< a + \sqrt{L\sigma_k}$. Hence it is in the tubular neigborhood and \[\varphi(k)(x) = \sigma_k^{-1} \Pi\left(\sigma_k \tilde{\varphi}(k)\bigl(\frac{x}{\rho}\bigr)\right)\] is well-defined. Furthermore it satisfies 
\begin{itemize}
\item[(i)] $\varphi(k)(x)=\sigma_k^{-1} u(k)(x)$ for $\abs{x}=\rho$, $\varphi(k)(x)=\sigma_k^{-1} \Pi(\sigma_k \tilde{a}(k)(\frac{x}{1-\lambda}))$ for $\abs{x}=(1-\lambda)\rho $;\\
\item[(ii)] $\int_{B_\rho\setminus B_{(1-\lambda)\rho}} \abs{D\varphi(k)}^2 \le C\rho^{N-2}\delta$.
\end{itemize}
Given competitors $c_l \in W^{1,2}(B_\rho, \A_Q(\R^m))$ to each $a_l$. As in \eqref{eq:5.101} set
\[c(k)= \sum_{l=1}^L c_l \oplus \sigma_k^{-1} p_l(k) \text{ and } \tilde{c}(k)=\sum_{l=1}^L \psi_k(a_l) \oplus \sigma_k^{-1} p_l(k).\]
As for $\tilde{a}(k)$ we have $\dist(\sigma_k \tilde{c}(k)(x), \A_Q(\MN))\le L\sigma_k$. Remark \ref{rem:2.101} (iv) states 
\[ (1-2 \sqrt{L\sigma_k} A)\abs{D(\Pi(\sigma_k \tilde{c}(k))}^2 \le \sigma_k^2 \abs{D\tilde{c}(k)}^2 \le \sigma_k^2 \abs{Dc(k)}^2 \]
with $A=\sup_{p \in \MN} \norm{A_p}_\infty$ and $\abs{D\tilde{c(k)}}^2=\sum_{l=1}^L \abs{D\psi_k(u_l)}^2 \le \sum_{l=1}^L \abs{Du_l}^2$ because $\psi_k$ is a $1-$Lipschitz retraction. We can define a competitor to $u(k)$ by 
\[ v(k)(x) = \begin{cases}
 u(k)(x), &\text{ for } \rho< \abs{x} \le  1\\
 \sigma_k \varphi(k)(x), &\text{ for } (1-\lambda)\rho< \abs{x} \le \rho \\
 \Pi(\sigma_k \tilde{c}(k)(\frac{x}{1-\lambda})), &\text{ for }  \abs{x}\le (1-\lambda)\rho.
 \end{cases}\]
Hence by minimality of each $u(k)$ we get
\[\int_{B_1} \abs{Du(k)}^2 \le \int_{B_1} \abs{Dv(k)}^2 \le \int_{B_1\setminus B_\rho} \abs{Du(k)}^2 + C \sigma_k^{2} \delta + \frac{\sigma_k^{2}(1-\lambda)^{N-2}}{1-2\sqrt{L\sigma_k}A} \int_{B_\rho}  \abs{Dc(k)}^2 \]
or 
\[ \sigma_k^{-2} \int_{B_\rho} \abs{Du(k)}^2 \le C\delta + \frac{1}{1-2\sqrt{L\sigma_k}A} \sum_{l=1}^L \int_{B_\rho} \abs{Dc_l}^2.\]
This implies that, by lower semicontinuity of the energy,
\begin{equation}\label{eq:5.102} \sum_{l=1}^L \int_{B_\rho} \abs{Da_l}^2 \le \liminf_{k \to \infty} \sigma_k^{-2} \int_{B_\rho} \abs{Du(k)}^2 \le C\delta + \sum_{l=1}^L \int_{B_\rho} \abs{Dc_l}^2. \end{equation}
$\delta>0$ can be taken arbitrary small, so each $a_l$ must be minimizing on $B_R \subset B_\rho \subset B_1$.  Choose $c_l=a_l$ for each $l$ in \eqref{eq:5.102} to deduce the strong convergence of energy, i.e. (iii).
\end{proof}

\begin{lemma}\label{lem:5.102}
There exists $\epsilon_0 >0$ and $\alpha>0, C>1$ depending on $N,Q, \MN$ with the property that, if $u\in W^{1,2}(\Omega, \A_Q(\MN))$ is energy minimizing with
\begin{equation}\label{eq:5.103}
E(u,B_{R_0}(y_0)) \le \epsilon_0 \text{ for some } B_{R_0}(y_0) \subset \Omega,
\end{equation}
then $\abs{Du}$ is an element of the Morrey space $L^{2,N-2+2\alpha}(B_{\frac{R_0}{2}}(y_0))$. More precisely we have the estimate 
\begin{equation}\label{eq:5.104}
E(u,B_r(y)) \le C \left(\frac{r}{R}\right)^{2\alpha} E(u,B_R(y)) \forall y \in \overline{B_{\frac{R_0}{2}}(y_0)}, 0< r\le R \le \frac{R_0}{2}.
\end{equation}
Furthermore $u \in C^{0,\alpha}(B_{\frac{R_0}{2}}(y_0))$.
\end{lemma}

\begin{proof}
First we will prove the following statement and show thereafter how it implies \eqref{eq:5.104}.\\

\emph{ $\exists \epsilon_1 >0, 0 < \gamma <1$ depending on $N, Q,\MN$ s.t. if $u\in W^{1,2}(B_R(y), \A_Q(\MN))$ is energy minimizing and $E(u,B_R(y))<\epsilon_1$ then
\begin{equation}\label{eq:5.105}
E(u,B_{\frac{R}{2}}(y)) < \gamma E(u,B_R(y)).
\end{equation}
}
Indeed, fix $\gamma<2^{-2\alpha_0}$, where $\alpha_0=\alpha_0(N,Q)>0$ is the H\"older exponent for Dirichlet minimizers into $\R^m$, compare \cite[Theorem 0.9]{Lellis}. Suppose such an $\epsilon_1>0$ does not exists, hence there are $v(k) \in W^{1,2}(B_{R_k}(y_k), \A_Q(\MN))$ energy minimizing failing \eqref{eq:5.105}, i.e. $E(v(k),B_{\frac{R_k}{2}}(y_k)) \ge \gamma E(v(k),B_{R_k}(y_k))$ and $\sigma_k^2=E(v(k),B_{R_k}(y_k) \to 0$  as $k \to \infty$. Consider the rescaled sequence \[ u(k)(x)= v(k)(y_k + R_k x) \text{ i.e. } E(u(k),B_1)=E(v(k),B_{R_k}(y_k))=\sigma_k^2.\] So we can apply the previous lemma \ref{lem:5.102}: for a subsequence $u(k)$, not relabeled, there are Dirichlet minimizing $a_l \in W^{1,2}(B_1, \A_{Q_l}(\R^m))$ ( $\sum_{l=1}^L Q_l =Q$) and a sequence of points $p_l(k) \in \MN$ such that for $a(k)=\sum_{l=1}^L a_l \oplus \sigma_k^{-1} p_l(k)$ one has
\begin{itemize}
\item[(i)] $\G(\sigma_k^{-1} u(k),a(k)) \to 0$ in $L^2(B_1)$;\\
\item[(ii)] $\lim_{k\to \infty} \sigma_k^{-2} E(u(k),B_R) = \sum_{l=1}^L E(a_l,B_R)$ for all $0<R<1$.
\end{itemize}
We firstly observe that this implies  $\sum_{l=1}^L E(a_l,B_{\frac{1}{2}}) \ge \gamma$ because \[\sigma_k^{-2} E(u(k), B_{\frac{1}{2}}) =\sigma_k^{-2} E(v(k),B_{\frac{R_k}{2}}(y_k)) \ge \gamma.\] Secondly 
\[ \sum_{l=1}^L E(a_l,B_{\frac{1}{2}}) = \lim_{k \to \infty} \sigma_k^{-2} E(u(k), B_{\frac{1}{2}}) \ge \liminf_{k\to \infty} \gamma \sigma_k^{-2} E(u(k),B_1) \ge \sum_{l=1}^L E(a_l,B_1).\]
So there must be a nontrivial $a_l$, with $E(a_l,B_{\frac{1}{2}}) \ge \gamma E(a_l,B_1)$. But $a_l$ is Dirichlet minimizing and therefore $E(a_l,B_{\frac{1}{2}}) \le 2^{-2\alpha_0} E(a_l,B_1)$. This is a contradiction. \\
Set $\epsilon_0=2^{-N} \epsilon_1 >0$, then \eqref{eq:5.104} holds because, if $E(u,B_{R_0}(y_0))<\epsilon_0$, then
\[ E(u,B_R(y))\le E(u,B_{\frac{R_0}{2}}(y)) \le 2^N E(u,B_{R_0}(y_0))\quad \forall y \in \overline{B_{\frac{R_0}{2}}(y_0)}, 0<R<\frac{R_0}{2},\]
as a consequence of the monotonicity formula \eqref{eq:2.107}.\\
Induction on \eqref{eq:5.105} gives $E(u, B_{2^{-k}R}(y)) \le \gamma^k E(u,B_R(y))$ for all $k \in \N$ and any $y \in  \overline{B_{\frac{R_0}{2}}(y_0)}, 0<R<\frac{R_0}{2}$. Choose $k\in \N$ s.t. $2^{-k-1}R<r \le 2^{-k}$ for $r<R$.  Then by monotonicity \eqref{eq:2.107} and the estimates above we have 
\[ E(u,B_r(y)) \le E(u,B_{2^{-k}R}(y)) \le \frac{1}{\gamma} \gamma^{k+1} E(u,B_R(y)) \le \frac{1}{\gamma} \left(\frac{r}{R}\right)^{2\alpha} E(u,B_R(y))\]
for $2\alpha= \frac{-\ln(\gamma)}{\ln(2)}$. \\
\eqref{eq:5.105} implies that $\abs{Du}$ is an element of the Morrey space $L^{2,N-2+2\alpha}(B_{\frac{R_0}{2}}(y_0))$. The H\"older continuity then follows classically. 
\end{proof}

\section{Properties of the singular set $\sing_H u$} 
\label{sec:properties_of_the_singular_set_sing_h_u_}

In this section let $u\in W^{1,2}(\Omega, \A_Q(\MN))$ be a fixed energy minimizing map. 
For any $B_{R_0}(y) \subset \Omega$, the monotonicity formula, \eqref{eq:2.107}, gives
\[
	\Theta_{u}(y)=\inf_{0< R \le R_0} E(u,B_R(y)) \le E(u,B_R(y)) \le E(u,B_{R_0}(y)) \quad \forall 0<R\le R_0.
\]
For any sequence $R_k \to 0$ and $y \in \Omega$ we may consider the rescaled sequence $v(k)(x)=u_{y,R_k}(x)=u(y+ R_k x)$ and observe that for any $r>0$, sufficient large $k\in \N$, i.e. $R_k \le \frac{R_0}{r}$
\[
	E(v(k),B_r)= E(u,B_{rR_k}(y)) \le E(u,B_{R_0}(y)).
\]
The compactness result, lemma \ref{lem:4.1}, asserts for a subsequence $v(k')$ there is $\varphi \in W^{1,2}(\R^N, \A_Q(\MN))$ energy minimizing with $\G(v(k),\varphi)\to 0$ in $L^2_{loc}(\R^N)$ and
\begin{equation}\label{eq:6.001}
	E(\varphi,B_r)=\lim_{k \to \infty } E(v(k),B_r)= \lim_{k \to \infty} E(u, B_{rR_k}(y))=\Theta_u(y) \quad \forall R>0.
\end{equation}
Furthermore the monotonicity formula, \eqref{eq:2.107} gives \[0=\int_{B_R\setminus B_r} \abs{x}^{2-N} \abs{\frac{\partial \varphi}{\partial r}}^2 \quad \forall 0<r<R.\] So, $0=\abs{\frac{\partial \varphi}{\partial r}}^2= \sum_{l=1}^Q \abs{\frac{\partial \varphi_l}{\partial r}}^2=0$ a.e.. Integrating this in $r$ gives $\varphi(\lambda x)=\varphi(x)$ for all $\lambda >0$ and $x \in \R^N$. This homogeneous degree zero property is characteristic for tangent maps, hence we define classically:

\begin{definition}\label{def:6.1}
A zero homogenous function $\varphi\in W^{1,2}(\R^N, \A_Q(\MN))$ is called tangent map to $u$ at $y\in \Omega$ if
\[ \exists R_k \to 0 \text{ with } \G(u_{y,R_k}, \varphi) \to 0 \text{ in } L^2_{loc.}(\R^N).\]
\end{definition}

\subsection{Properties of homogeneous degree zero minimizers} 
\label{sub:properties_of_homogeneous_degree_zero_minimizers}
Let us consider $\varphi \in W^{1,2}(\R^N, \A_Q(\MN))$ be energy minimizing and zero homogeneous, i.e. $\varphi(\lambda x)=\varphi$ for all $x \in \R^N, \lambda>0$. Every tangent map, definition \ref{def:6.1}, has this property. In this section we  state some consequences. First of all one observe that the multivalued case does not differ from the single valued, "classical" case. Our presentation follows very closely L.Simon's in \cite[section 3]{Simon}. The analysis of tangent maps enables a stratification procedure, section \ref{sub:consequences_for_sing_h_u_}. It is a direct modification of a result by F.~Almgren, \cite{Almgren2}. As a consequence we will be able to get an estimate on the singular set $\sing_H u$.
\begin{equation}\label{eq:6.201}
\Theta_\varphi(y) \text{ takes its maximum in } y=0.
\end{equation}

Indeed, fix $y \in \R^N$, for any $0<R$ combining the monotonicity \eqref{eq:2.109} with $E(\varphi, B_r(0))=\Theta_\varphi(0)\; \forall r>0$ gives
\begin{align*}
	& 2 \int_{B_R(y)} \abs{x-y}^{2-N} \abs{\frac{\partial \varphi}{\partial r_y}}^2 + \Theta_\varphi(y) = E(\varphi, B_R(y))\\
	 & \le \left(1 + \frac{\abs{y}}{R}\right)^{N-2} E(\varphi,B_{R+ \abs{y}}(0))=\left(1 + \frac{\abs{y}}{R}\right)^{N-2} \Theta_u(0);
\end{align*}
with $\frac{\partial}{\partial r_y}$ we want to emphasize the center $y$ i.e. it is the directional derivative in the radial direction $\frac{x-y}{\abs{x-y}}$. Taking the limit $R \to \infty$ we get
\begin{equation}\label{eq:6.202}
	2 \int_{\R^N} \abs{x-y}^{2-N} \abs{\frac{\partial \varphi}{\partial r_y}}^2 + \Theta_\varphi(y) \le \Theta_u(0)= \Theta_\varphi(0).
\end{equation}\\

\begin{definition}\label{def:6.201}
Let $\varphi \in W^{1,2}(\R^m,\A_Q(N))$ be a homogeneous degree $0$ energy minimizer. Then we define
\[
	S(\varphi)=\{ y \in \R^N \colon \Theta_\varphi(y)=\Theta_\varphi(0) \}.
\] 
\end{definition}
We next claim that
\begin{align}\label{eq:203}
S(\varphi) &\text{ is a linear subspace of } \R^N\\ \label{eq:2031}
\text{ and }\varphi(x+y) = \varphi(x) &\text{ for all } x\in \R^N, y \in S(\varphi)
\end{align}
To show \eqref{eq:203} and \eqref{eq:2031} observe that for $y \in S(\varphi)$, equality in \eqref{eq:6.202} implies $\frac{\partial \varphi}{\partial r_y}=0$ i.e.
\[ \varphi(y + \lambda x)= \varphi(y +x) \quad \forall x \in \R^N\lambda >0\]
Combing this with, $\varphi(\tilde{\lambda}x) =\varphi(x)\quad \forall x\in \R^N, \tilde{\lambda}>0$ gives
\begin{align*}
	\varphi(x)&=\varphi(\lambda x)\\
	 &= \varphi( y + (\lambda x -y )) = \varphi( y + \lambda^{-2}(\lambda x -y))= \varphi(\lambda^{-1} x + (y -\lambda^{-2} y)) \\
	&= \varphi(x + (\lambda - \lambda^{-1})y)= \varphi(x + \mu y)
\end{align*}
where $\mu= \lambda-\lambda^{-1}$ is an arbitrary real number. This implies naturally $E(u,B_R(0))=E(u,B_R(\mu y))$ and $\Theta_\varphi(0)= \Theta_\varphi(\mu\, y)$ for all $\mu \in \R$ and $y \in S(\varphi)$.\\

\subsection{Consequences for $\sing_H u$} 
\label{sub:consequences_for_sing_h_u_}
The obtained results gives us equivalent identifications of the H\"older regular set.
\begin{lemma}\label{lem:6.401}
Let $u\in W^{1,2}(\Omega, \A_Q(\MN))$ be energy minimizing, then the following are equivalent
\begin{itemize}
	\item[(i)] $y \in \reg_H u$;
	\item[(ii)] $\Theta_u(y)=0$;
	\item[(iii)] $u$ has a constant tangent map $\varphi$ at $y$;
	\item[(iv)] $\operatorname{ dim } S(\varphi) =N$ for some tangent map $\varphi$ of $u$ at $y$.
\end{itemize}
\end{lemma}

\begin{proof}
\emph{ (i) $\Rightarrow$ (iii):} Let $\varphi$ be any tangent map of $u$ at $y$. Passing to a subsequence we have $u_{y,R_k}(x)=u(y + R_k x)$ converging locally a.e. to $\varphi$. Hence for a.e. $x,x'$ we have
\[
	\G(\varphi(x),\varphi(x'))= \lim_{k \to \infty} \G(u(y+ R_k x), u(y+ R_k x')) \le \liminf_{k \to \infty} C R_k^\alpha \abs{x-x'} =0.
\]
Thus $\varphi \equiv const.$.\\
\emph{ (ii) $\Leftrightarrow$ (iii) :} This equivalence is obvious.\\
\emph{ (iii) $\Leftrightarrow$ (iv) :} This equivalence just follows by definition and the last observation in the previous section.\\
\emph{ (ii) $\Rightarrow$ (i) :} If $\Theta_u(y)=0$ there is a $R>0$ s.t. $E(u,B_R(y))<\epsilon_0$, where $\epsilon_0>0$ is the constant of lemma \ref{lem:5.101}. Then this lemma states $u \in C^{0,\alpha}(B_\frac{R}{2}(y))$ and so $y \in \reg_Hu$.
\end{proof}

\begin{remark}\label{rem:6.401}
For single valued, "classical" harmonic functions, lemma \ref{lem:6.401} implies
\[ \reg u = \reg_H u \text{ and so } \sing u = \sing_H u \].
\end{remark}

Furthermore lemma \ref{lem:6.401} has the following simple consequences as in the single valued setting. 

\begin{lemma}\label{lem:6.402}
$\mathcal{H}^{N-2}(\sing_H u) =0$
\end{lemma}
\begin{proof}
This is a classical consequence of $\abs{Du}^2$ being in $L^1$ and $\sing_H u = \{ y\,:\, \Theta_u(y) > \epsilon_0 \}$.
\end{proof}

One defines
\begin{equation}\label{eq:6.401}
S_j=\{ y \in \sing_H u \colon \operatorname{dim}S(\varphi) \le j \text{ for all tangent maps $\varphi$ at $y$} \}. 
\end{equation}

We first observe that
\begin{equation}\label{eq:6.402}
	\sing_H u = S_{N-1}= S_{N-2} =S_{N-3}.
\end{equation}
Indeed, suppose not. Then there would be a tangent map $\varphi$, which is a non constant homogenous degree zero minimizer with $N-1 \ge \operatorname{dim}S(\varphi) \ge N-2$. This contradicts lemma \ref{lem:6.402} because $S(\varphi) \subset  \sing_H \varphi$ and 
\[
	+ \infty = \mathcal{H}^{N-2}( S(\varphi)) = \mathcal{H}^{N-2}(\sing_H u).
\]

As L.~Simon mentions in \cite[section 3.4]{Simon} one notice:
\begin{quote}
"The subsets $S_j$ are mainly important because of the following lemma, which is a direct modification of the corresponding result for minimal surfaces by F.~Almgren \cite{Almgren2}; the lemma can be thought of as a refinement of the "dimension reducing" argument of Federer \cite{Federer} (for this see also the discussion in the appendix of \cite{Simon2}). "\footnote{L.~Simon, \cite{Simon}, page 54}
\end{quote}

Classically a characterization of $S_j$ implies a $\delta$- approximation property which then itself implies the following two results. Their classical proofs can be found in \cite[section 3.4, Lemma 1 \& Corollary 1]{Simon}
\begin{lemma}\label{lem:6.403}
For each $j=0, \dotsc, N-3$, $\dim S_j \le j$, and for each $t > 0$, $S_0 \cap \{y \, :\, \Theta_u(y)=t \} $ is a discrete set.
\end{lemma}
\begin{corollary}\label{cor:6.404}
$\operatorname{ dim } \sing_H u \le N-3$. More generally, if all tangent maps $\varphi$ of $u$ satisfy $\operatorname{dim}S(\varphi) \le j_0 \le N-3$ then $\operatorname{dim} \sing_H u \le j_0$. 
\end{corollary}
This corollary clearly shows theorem \ref{theo:0.1}.

\begin{appendix}
\section{Q-valued functions} 
\label{sec_I:Q-valued functions}
As announced in the introduction we recall the basic definitions and results on $Q$-valued functions needed in the article. The theory is presented omitting the actual proofs. They can be found for instance in C.~De Lellis and E.~Spadaro's work \cite{Lellis}. 

As mentioned we follow mainly the notation and terminology introduced by C.~De Lellis and E.~Spadaro in \cite{Lellis}. It differs slightly from Almgren's original one. $Q, Q_1, Q_2, \dotsc$ are always natural numbers.\\
The space of unordered sets of $Q$ points in $\R^n$ can be made into a complete metric space.
\begin{definition}\label{def_I:1.101}
	$\left(\A_Q(\R^n), \G \right)$ denotes the metric space of unordered $Q$-tuples given by
	\begin{equation*}
	\mathcal{A}_Q(\R^n)= \left\{ T= \sum_{i=1}^Q \llbracket t_i \rrbracket \colon t_i \in \R^n, i = 1, \dotsc, Q \right\}
	\end{equation*}
	and if $\mathcal{P}_Q$ is the permutation group of $\{1, \dotsc, Q\}$ the metric is given by
	\begin{equation*}
		\G(S,T)^2= \min_{\sigma \in \mathcal{P}_Q} \sum_{i=1}^Q \abs{s_i - t_{\sigma(i)}}^2.
	\end{equation*}
\end{definition}
We use the convention $\llbracket t \rrbracket = \delta_t$ for a Dirac measure at a point $t \in \R^n$. Considering $T=\sum_{i=1}^Q \llbracket t_i \rrbracket$ as a sum of $Q$ Dirac measures one notice that $\A_Q(\R^n)$ corresponds to the set of  $0$-dimensional integral currents of mass $Q$ and positive orientation. Hence we will write 
\[
	spt(T)=\{ t_1, \dots, t_Q \colon T=\sum_{i=1}^Q \llbracket t_i \rrbracket \} \subset \R^n.
\]
 
Furthermore $\A_Q(\R^n)$ is endowed with an intrinsic addition:
\[
	+\colon \A_{Q_1}(\R^n) \times \A_{Q_2}(\R^n) \to \A_{Q_1+Q_2}(\R^n) \quad S+T = \sum_{i=1}^{Q_1} \llbracket s_i \rrbracket + \sum_{i=1}^{Q_2} \llbracket t_i \rrbracket.
\]
We define a translation operator
\[
	\oplus\colon A_Q(\R^n) \times \R^n \to \A_Q(\R^n) \quad T\oplus s = \sum_{i=1}^Q \llbracket t_i + s \rrbracket.
\]
The metric $\G$ defines continuity, modulus of continuity, H\"older and Lipschitz continuity and (Lebesgue) measurability for functions from a set $\Omega \subset \R^N$ into $\A_Q(\R^n)$, i.e.$u: \Omega \to \A_Q(\R^n)$. \\
As it has been shown in \cite[Proposition 0.4]{Lellis} for any measurable function $u: \Omega\to \A_Q(\R^n)$ we can find a measurable selection i.e.
\[
	v=(v_1, \dotsc , v_Q): \Omega \to (\R^n)^Q \text{ measurable s.t. } u(x)=[v](x)=\sum_{i=1}^Q \llbracket v_i(x) \rrbracket.
\]
Selections of higher regularity are considered in \cite{Lellis select}, \cite[Proposition 1.2]{Lellis} and in the appendix to \cite{H1}.\\
We will write $\abs{u(x)}=\sqrt{\sum_{i=1}^Q \abs{v_i(x)}^2}=\G(u(x),Q\llbracket 0 \rrbracket)$.
\begin{definition}\label{def_I:1.102}
	The Sobolev space $W^{1,2}(\Omega, \A_Q(\R^n))$ is defined as the set of measurable functions $u: \Omega\subset \to \A_Q(\R^n)$ that satisfy 
	\begin{itemize}
		\item[(w1)] $x \mapsto \G(u(x),T) \in W^{1,2}(\Omega, \R_+)$ for every $T \in \A_Q(\R^n)$;
		\item[(w2)] $\exists \varphi_j \in L^2(\Omega, \R_+)$ for $j=1, \dots, N$ s.t. $\abs{D_j\G(u(x),T)} \le \varphi_j(x)$ for any $T \in \A_Q(\R^n)$ and a.e. $x \in \Omega$.
	\end{itemize}
\end{definition}
It is not difficult to show the existence of minimal functions $\tilde{\varphi}_j$, in the sense that $\tilde{\varphi}_j(x)\le \varphi_j(x)$ for a.e. $x$ and any $\varphi_j$ satisfying property (w2), \cite[Proposition 4.2]{Lellis}. Such a minimal bound is denoted by $\abs{D_ju}$ and is explicitly characterised by
\[
	\abs{D_ju}(x)= \sup\left\{ \abs{D_j \G(u(x), T_i)}\colon \{T_i\}_{i\in \N} \text{ dense in } \A_Q(\R^n)\right\}.
\]
The Sobolev "semi-norm", or Dirichlet energy, is defined by integrating the measurable function $\abs{Du}^2=\sum_{j=1}^N \abs{D_ju}^2$:
\begin{equation}\label{eq_I:1.101}
	\int_\Omega \abs{Du}^2 = \int_{\Omega} \sum_{j=1}^J \abs{D_ju}^2.
\end{equation} 
Strictly speaking it is not a "semi-norm". $W^{1,2}(\Omega, \A_Q(\R^n))$ is not a linear space since $\A_Q(\R^n)$ lacks this property.\\
A function $u \in W^{1,2}(\Omega, \R^n)$ is said to be Dirichlet minimizing if
\begin{equation}\label{eq_I:1.102}
	\int_{\Omega} \abs{Du}^2= \inf \left\{ \int_{\Omega} \abs{Dv}^2 \colon v \in W^{1,2}(\Omega, \A_Q(\R^n)), \G(u(x), v(x)) \in W^{1,2}_0(\Omega, \R_+) \right\}.
\end{equation}

On Lipschitz regular domains $\Omega \subset \R^N$ one has a continuous trace operator as for classical single valued Sobolev functions
\[
	\circ\tr{\partial \Omega}: W^{1,2}(\Omega, \A_Q(\R^n)) \to L^2(\partial \Omega, \A_Q(\R^n)).
\]
The definition of $W^{1,2}(\Omega, \A_Q(\R^n))$, definition \ref{def_I:1.102}, implies that on a Lipschitz regular domain $\Omega \subset \R^N$ one has that $\G(u(x),v(x)) \in W^{1,2}_0(\Omega)$ corresponds to $u\tr{\partial \Omega}=v\tr{\partial \Omega}$ for any $u,v \in W^{1,2}(\Omega, \A_Q(\R^n))$.\\

As a consequence of a Rademacher theorem for multivalued Lipschitz functions, \cite[section 1.3 \& Theorem 1.13]{Lellis} a Sobolev function $u \in W^{1,2}(\Omega, \A_Q(\R^n))$ is a.e. approximately differentiable in the sense
\begin{itemize}
	\item[(1)] $\exists \mathcal{U}_x: \Omega \to \A_Q(\R^n \times Hom(\R^N, \R^n))$, $x \mapsto \mathcal{U}_x= \sum_{i=1}^Q \llbracket (u_i(x), U_i(x)) \rrbracket$ measurable with $U_i(x)=U_j(x)$ whenever $u_i(x)=u_j(x)$;
	\item[(2)] $\mathcal{U}_x$ defines a 1-jet $J\mathcal{U}_x: \Omega \times \R^N \to \A_Q(\R^n)$ by $J\mathcal{U}_x(y)=\sum_{i=1}^Q \llbracket u_i(x) + U_i(x)(y-x) \rrbracket$, that has the additional property that $J\mathcal{U}_x(x)= u(x)$ for a.e. $x \in \Omega$;\\
	\item[(3)] for a.e. $x\in \Omega$,  $\exists E_x \subset \Omega$ having density $1$ in $x$ s.t.  $\G(u(y), J\mathcal{U}_x(y))=o(\abs{y-x})$ on $E_x$. 
\end{itemize}
As one may guess the 1-jet corresponds to a first order "Taylor expansion", that becomes apparent in the proof of Rademacher's theorem, \cite[Theorem 1.13]{Lellis}. 
One can show that $\abs{D_ju}(x)= \sum_{i=1}^Q \abs{U_i(x)e_j}^2$ for a.e. $x \in \Omega$, \cite[Proposition 2.17]{Lellis}. From now on we will write $Du_i(x)$ for $U_i(x)$ and $D_ju_i(x)$ for $U_i(x)e_j$.\\

A useful tool is Almgren's bi-Lipschitz embedding of $\A_Q(\R^n)$ into some $\R^N$. A remark of Brian White improved it, compare \cite[Theorem 2.1 \& Corollary 2.2]{Lellis}:

\begin{theorem}[bi-Lipschitz embedding]\label{theo_I:1.101}
	There exists $m=m(Q,n)$ and an injective map $\boldsymbol{\xi}: \A_Q(\R^n) \to \R^m$ with the properties
	\begin{itemize}
		\item[(i)] $Lip(\boldsymbol{\xi})\le 1$ and $Lip(\boldsymbol{\xi}^{-1}\vert_{\boldsymbol{\xi}(\A_Q(\R^n))})\le C(Q,n)$;
		\item[(ii)]  $\forall T \in \A_Q(\R^n)$ $\exists\delta=\delta(T) >0$ such that $\abs{\boldsymbol{\xi}(T)-\boldsymbol{\xi}(S)}=\G(T,S)$ for all $S \in B_{\delta}(T) \subset \A_Q(\R^n)$.
	\end{itemize}
There is a retraction $\boldsymbol{\rho}: \R^m \to \A_Q(\R^n)$ because of (i) and the Lipschitz extension Theorem, e.g. \cite[Theorem 1.7]{Lellis}.
\end{theorem}
As a consequence $\abs{Du}(x)=\abs{D\boldsymbol{\xi}\circ u}(x)$ for a.e. $x \in \Omega$ for any $u \in W^{1,2}(\Omega, \mathcal{A}_Q(\R^n))$.\\
We want to remark that the image of $\A_Q(\R^n)$ under $\boldsymbol{\xi}$ in $\R^m$ is not convex neither a $C^2$ manifold. Thus there is no "nearest point" projection not even in a tubular neighborhood.

Two cornerstones in the context of Dirichlet minimizers mapping into $\R^n$ that are of interest for us are (c.p. with \cite[Theorem 0.8 \& Theorem 0.9]{Lellis}):
.
\begin{theorem}[Existence of Dirichlet minimizers]\label{theo_I:1.102}
	Let $v\in W^{1,2}(\Omega, \A_Q(\R^n))$ be given, then there exists a (not necessarily unique) Dirichlet minimizing $u\in W^{1,2}(\Omega, \A_Q(\R^n))$ with $\G(u(x),v(x)) \in W^{1,2}_0(\Omega,\R_+)$.  
\end{theorem}
and the already stated 
\begin{theorem}[interior H\"older continuity]\label{theo_I:1.103}
There is a constant $\alpha_0=\alpha_0(N,Q)>0$ with the property that if $u \in W^{1,2}(\Omega, \A_Q(\R^n))$ is Dirichlet minimizing, then $u\in C^{0,\alpha_0}(K, \A_Q(\R^n))$ for any $K\subset \Omega\subset \R^N$ compact. Indeed, $\abs{Du}$ is an element of the Morrey space $L^{2,N-2-2\alpha_0}$ with the estimate
\begin{equation}\label{eq_I:1.103}
	r^{2-N-2\alpha_0} \int_{B_r(x)} \abs{Du}^2 \le R^{2-N-2\alpha_0} \int_{B_R(x)} \abs{Du}^2 \text{ for } r\le R, B_R(x) \subset \Omega.
\end{equation}
For two-dimensional domains $\alpha_0(2,Q)=\frac{1}{Q}$ is explicit and optimal.
\end{theorem}
Both results had been proven first by Almgren in \cite{Almgren} and nicely reviewed by C. De Lellis and E. Spadaro in \cite{Lellis}.\\

J.~Almgren presents in  \cite[Theorem 2.16]{Almgren} an example of non-uniqueness: there are two Dirichlet minimizers $f \neq h\in W^{1,2}(B_1, \A_2(\R^2))$, $B_1 \subset \R^2$, with $f = h$ on $\partial B_1$. Given any other minimzer that agrees with $f$ or $h$ at the boundary must be either $f$ or $h$.

\section{Concentration compactness for $Q$-valued functions}\label{sec:concentration_compactness_for_q_valued_functions}
Let $\Omega \subset \R^N$ be given, then there is a concentration compactness lemma for sequences $u(k) \in W^{1,2}(\Omega,\A_Q(\R^n))$ with uniformly bounded energy.

\begin{lemma}\label{lem_I:A1.1}
Given a sequence $u(k) \in W^{1,2}(\Omega, \A_Q(R^n))$ and a sequence of means $T(k) \in \A_Q(\R^n)$ with
\[
	\limsup_{k \to \infty} \int_{\Omega} \abs{Du(k)}^2 \le \infty \text{ and } \int_{\Omega} \G(u(k),T(k))^2 \le C \int_{\Omega} \abs{Du(k)}^2
\]
for a subsequence, not relabelled, we can find:
\begin{itemize}
	\item[(i)] maps $b_l \in W^{1,2}(\Omega, \A_{Q_l}(\R^n))$ for $l=1, \dotsc , J$, $\sum_{l=1}^L Q_l =Q$;\\
	\item[(ii)] a splitting $T(k) = T_1(k) + \dotsm + T_L(k) $ with $T_l(k) \in \A_{Q_l}(\R^n)$ and
	\begin{itemize}
		\item $\limsup_{k} diam(spt(T_l(k)))< \infty$ for all $l=1, \dotsc, L$\\
		\item $\lim_{k \to \infty} \dist(spt(T_l(k)), spt(T_{m}(k)))= \infty$ for $l\neq m$;
	\end{itemize}
	\item[(iii)] a sequence $t_l(k) \in spt(T_l(k))$ such that $\G(u(k), b(k)) \to 0$ in $L^2$ with $b(k)= \sum_{l=1}^L  (b_l \oplus t_l(k))$.
\end{itemize}
Moreover, the following two additional properties hold:
\begin{itemize}
	\item[(a)] if $\Omega' \subset \Omega$ is open and $A_k$ is a sequence of measurable sets with $\abs{A_k} \to 0$, then
	\[
		\liminf_{k \to \infty} \int_{\Omega'\setminus A_k} \abs{Du(k)}^2 - \int_{\Omega'} \abs{Db(k)}^2 \ge 0.
	\]
	\item[(b)] $\liminf_{k \to \infty} \int_{\Omega} \left(\abs{Du(k)}^2 -\abs{Db(k)}^2\right) = 0$ if and only if\\ $\liminf_{k \to \infty} \int_{\Omega} \bigl(\abs{Du(k)} - \abs{Db(k)})^2 = 0$.
\end{itemize}
\end{lemma}

Before we give the proof we recall the definition of the separation $sep(T)$ of a $Q$-point $T=\sum_{i=1}^Q \llbracket t_i \rrbracket \in \A_Q(\R^n)$.
\[
sep(T)=\begin{cases}
0, &\text{ if } T=Q\llbracket t \rrbracket \\
\min_{t_i \neq t_j} \abs{t_i - t_j}, &\text{ otherwise }.
\end{cases}\]

The following results are of essential use in the context of the separation and needed for the proof of the concentration compactness lemma. The first gives a kind of relation between $diam(spt(T))$ and $sep(T)$, see \cite[lemma 3.8]{Lellis}; the second gives a retraction $\boldsymbol{\vartheta}=\boldsymbol{\vartheta}_T$ based on $sep(T)$, see \cite[lemma 3.7]{Lellis}

\begin{lemma}\label{lem_I:A1.2}
To every $\epsilon>0$ there exists $\beta=\beta(\epsilon, Q)>0$ with the property that to any $T \in \A_Q(\R^n)$ there exists $S=S(T) \in \A_Q(\R^n)$ with
\[ spt(S) \subset spt(T), \quad \G(T,S) < \epsilon\, sep(S) \text{ and } \beta \, diam(spt(T)) < sep(S). \]
(For example $\beta=\epsilon^Q \, 3^{4-Q^2}$ works.)
\end{lemma}

\begin{lemma}\label{lem_I:A1.3}
To a given $T\in \A_Q(\R^n$ and $0< 4s < sep(T)$ there exists a $1-$Lipschitz retraction 
\[\boldsymbol{\vartheta}=\boldsymbol{\vartheta}_T: \A_Q(\R^n) \to \overline{B_s(T)}=\{ S \in \A_Q(T) \colon \G(S,T) \le s \} \]
with the property that
\begin{itemize}
\item[(i)] $\boldsymbol{\vartheta}(S)=S$ if $\G(S,T) \le s$;\\
\item[(ii)] $\G(\boldsymbol{\vartheta}(S_1), \boldsymbol{\vartheta}(S_2))<\G(S_1,S_2)$ if $\G(S_1,T)> s$.
\end{itemize}

\end{lemma}

\begin{proof}[Proof of lemma \ref{lem_I:A1.1}]
We distinguish two cases. The second will be handled by induction on the first.\\

\emph{Case 1 and basis of the induction: $\liminf_{k \to \infty} diam(spt(T(k)))< \infty$\\ ( $diam(spt(T(k)))=0$ for $Q=1$):}\\
Passing to an appropriate subsequence, not relabelled $diam(spt(T(k)))< C$ for all $k$. Set $L=1$, and as splitting keep the sequence itself i.e. $T(k) =T_1(k)$. To every $k$ fix a $t_1(k) \in spt (T(k))$.\\
Hence we have
\begin{align*}
	&\limsup_k \int_{\Omega} \abs{ u(k) \oplus  (-t_{1}(k))}^2 = \limsup_k \int_{\Omega} \G(u(k), Q \llbracket t_1(k) \rrbracket )^2\\
	 \le &\limsup_k 2 \int_{\Omega} \G(u(k), T(k))^2 + 2 \abs{\Omega} \G(T(k), Q\llbracket t_1(k) \rrbracket)^2 < \infty.
\end{align*}
Hence passing to an appropriate subsequence there is $b=b_1 \in W^{1,2}(\Omega, \A_Q(\R^n))$ with $u(k)\oplus(-t_1(k)) \to b$ in $L^2$. This proves (i),(ii),(iii), since $\G(u(k)\oplus- t_1(k), b)= \G(u(k), b\oplus t_1(k)) = \G(u(k), b(k))$. Furthermore, the established properties imply that $\boldsymbol{\xi} \circ u(k) \rightharpoonup \boldsymbol{\xi}\circ b(k)$ in $W^{1,2}(\Omega, \R^m)$. The additional property (a) follows, because $\mathbf{1}_{\Omega'\setminus A_k} \to \mathbf{1}_{\Omega'}$ in $L^2(\Omega)$ and so $\mathbf{1}_{\Omega'\setminus A_k} D\boldsymbol{\xi} \circ u(k) \rightharpoonup \mathbf{1}_{\Omega'}D\boldsymbol{\xi}\circ b(k)$. Property (b) holds because $L^2(\Omega)$ is an Hilbert space. Therefore we have, that $f_k = D \boldsymbol{\xi} \circ u(k) \to f= D\boldsymbol{\xi}\circ b(k)$ in $L^2(\Omega)$ if and only if $f_k \rightharpoonup f$ and $\norm{f_k}^2_{L^2(\Omega)} \to \norm{f}^2_{L^2(\Omega)}$; compare $\liminf_{k} \norm{f_k-f}^2 = \liminf_{k} \norm{f_k}^2 + \norm{f}^2 - 2 \langle f_k, f \rangle = \liminf_{k} \norm{f_k}^2 - \norm{f}^2$. \\

\emph{Case 2 and the induction step: $\liminf_k diam(spt(T(k)))= +\infty$}\\
Suppose the lemma holds for $Q'<Q$. To every $T(k)$ pick $S(k) \in \A_Q(\R^n)$ using \ref{lem_I:A1.2} s.t. for $S(k)=\sum_{j=1}^{J(k)} Q_j(k) \llbracket s_j(k) \rrbracket \in \A_Q(\R^n)$ set $\sigma_k = sep(S(k))$, then $\beta(\frac{1}{10},Q)\, diam(spt(T(k)))< \sigma_k$ and $\G(T(k),S(k)) < \frac{\sigma_k}{10}$. Passing to an appropriate subsequence, not relabelled, we may further assume that $J(k)>1$ and$Q_j(k)$ do not depend on $k$. Fix the associated 1-Lipschitz retractions of \ref{lem_I:A1.3} $\boldsymbol{\vartheta}_k: \A_Q(\R^N) \to \overline{B_{\frac{1}{5} s(S(k))}( S(k))}$ i.e. $\mathcal{H}^0\left( spt(\boldsymbol{\vartheta}_k(T))\cap B_{\frac{\sigma_k}{5}(s_j)}\right)=Q_j$ for all $T \in \A_Q(\R^n)$ and $j=1, \dotsc, J$. Hence these retractions $\boldsymbol{\vartheta}_k$ defines new sequences $v_j(k)$ in $W^{1,2}(\Omega, \A_{Q_j}(\R^n))$ and a splitting of $T(k)$:
\begin{align*}
	&\boldsymbol{\vartheta}_k \circ u(k) = v_1(k)+ \dotsb v_J(k) \text{ with } v_j(k) \in B_{\frac{\sigma_k}{5}}(s_j); \\
	T(k)=&\boldsymbol{\vartheta}_k \circ T(k) = T_1(k)+ \dotsb + T_J(k) \text{ with } T_j(k) \in B_{\frac{\sigma_k}{5}}(s_j)
\end{align*}
Each sequence $v_j(k)$, $j=1, \dotsc, J$ satisfies itself the assumptions of the lemma, because $\boldsymbol{\vartheta_k}$ is a retraction and so
\begin{align}
	\sum_{j=1}^J \abs{Dv_j(k)}^2 &= \abs{D\boldsymbol{\vartheta_k} \circ u(k)}^2 \le \abs{Du(k)}^2\label{eq_I:A1.1}\\
	\sum_{j=1}^J\G(v_j(k), T_j(k))^2 &= \G(\boldsymbol{\vartheta}_k\circ u(k), \boldsymbol{\vartheta}_k \circ T(k))^2 \le \G(u(k),T(k))^2.\label{eq_I:A1.2}
\end{align}
Furthermore we record some properties:\\
Defining  $A_k=\{x\,:\, \boldsymbol{\vartheta}_k \circ u(k)(x) \neq u(k)(x) \} = \{ x\, :\, \G(u(k),S(k)) > \frac{\sigma_k}{5} \} \subset \{x\,:\, \G(u(k),T(k)) \ge \frac{\sigma_k}{10} \}= B_k$ (subsets of $\Omega$) we have
\begin{itemize}
\item[(1.)] $\abs{B_k} \to 0$ as $k \to \infty$, because
\begin{align*}
\abs{B_k} &\le \left( \frac{10}{\sigma_k)}\right)^{2^*} \int_{B_k} \G(u(k),T(k))^{2^*}\\
&\le \left( \frac{10}{\sigma_k}\right)^{2^*} C \left(\int_{\Omega} \abs{Du(k)}^2\right)^{\frac{2^*}{2}} \to 0;
\end{align*}
\item[(2.)] $\G(u(k), \boldsymbol{\vartheta}_k \circ u(k)) \to 0$ in $L^2$ as $k \to \infty$, since
\begin{align*}
	&\int_\Omega \G(u(k),\boldsymbol{\vartheta}_k\circ u(k))^2 = \int_{A_k} \G(u(k), \boldsymbol{\vartheta}_k \circ u(k))^2\\
	&\le 2 \int_{B_k} \G(v_k, T(k))^2 + \G(\boldsymbol{\vartheta}_k\circ u(k), \boldsymbol{\vartheta}_k \circ T(k))^2\\
	&\le 4 \left(\frac{10}{\sigma_k}\right)^{2^*-2} \int_{B_k} \G(u(k), T(k))^{2^*}\\
	&\le \frac{C}{\sigma_k^{2^*-2}} \left(\int_{\Omega} \abs{Du(k)}^2\right)^{\frac{2^*}{2}} \to 0;
\end{align*}
\item[(3.)] $\dist(spt(T_i), spt(T_j)) \ge \sigma_k - 2 \G(S(k),T(k)) \ge \frac{4}{5} \sigma_k \to +\infty$ for any $i \neq j$ as $k \to \infty$;
\item[(4.)] $\abs{\abs{Du(k)} - \abs{D\boldsymbol{\vartheta}_k\circ u(k)}} \to 0$ in $L^2$ as $k \to \infty$, because $\abs{B_k} \to 0$, $\abs{D\boldsymbol{\vartheta}_k\circ u(k)}\le \abs{Du(k)}$, $D \boldsymbol{\vartheta}_k \circ u(k) = D u(k)$ on $\Omega\setminus B_k$ and
\begin{align*}
	&\int_{\Omega} \left( \abs{Du(k)} - \abs{D\boldsymbol{\vartheta}_k \circ u(k)}\right)^2 \le \int_{\Omega} \abs{Du(k)}^2- \abs{D\boldsymbol{\vartheta}_k \circ u(k)}^2\\
	 &= \int_{B_k} \abs{Du(k)}^2- \abs{D\boldsymbol{\vartheta}_k \circ u(k)}^2 \le \int_{B_k} \abs{Du(k)}^2\to 0.
\end{align*}
\end{itemize}
Due to the induction hypothesis the lemma holds for each sequence $v_j(k)$ i.e. we can find $b_{j,l} \in W^{1,2}(\Omega, \A_{Q_{j,l}}(\R^n))$, with $\sum_{l=1}^{L_j} Q_{j,l}=Q_j$, a splitting $T_j(k)= T_{j,1}(k) + \dotsb + T_{j,L_j}(k)$ together with sequences $t_{j,l}(k) \in spt(T_{j,l}(k))$ satisfying the conditions (i), (ii), (iii). Furthermore the additional properties (a),(b) hold. Set $L= \sum_{j=1}^J L_j$, $K_j=\sum_{i=1}^{j-1} L_i$ and relabel $b_{K_j+l}=b_{j,l}$, $T_{K_j+l}(k)=T_{j,l}(k)$, $t_{K_j+l}(k)=t_{j,l}(k)$ and $Q_{K_j+l}=Q_{j,l}$ for $j \in \{1, \dotsc, J\}$ and $l \in \{1, \dotsc, L_j\}$. The induction hypothesis on the lemma states that the obtained sequences $b_l$, $T_l(k)$, $t_l(k)$ for $l=1, \dotsc, L$ satisfy
\begin{itemize}
	\item[(i)] $b_l \in W^{1,2}(\Omega, \A_{Q_l}(\R^n))$ for $l=1, \dotsc , L$ and $\sum_{l=1}^{L} Q_l =Q$;\\
	\item[(ii)] $T(k) = T_1(k) + \dotsb + T_{L}(k)$, $t_l(k) \in spt( T_l(k))$ and
	\begin{itemize}
		\item $\limsup_{k} diam(spt(T_l(k)))< \infty$ for all $l=1, \dotsc , L$\\
		\item $\lim_{k \to \infty} \dist(spt(T_l(k)), spt(T_{m}))= \infty$ for $l\neq m$ for any $K_j < l < m \le K_{j+1}$, $j=1, \dotsc, J$
	\end{itemize}
	\item[(iii)] $\G(v_j(k), b_j(k)) \to 0$ in $L^2$ with $b_j(k)= \sum_{l=K_j+1}^{K_{j+1}} (b_l \oplus t_l(k)) $ for each $j$.
\end{itemize}
Moreover, the following two additional properties hold for each $j$:
\begin{itemize}
	\item[(a)] if $\Omega' \subset \Omega$ is open and $A_k$ is a sequence of measurable sets with $\abs{A_k} \to 0$, then
	\[
		\liminf_{k \to \infty} \int_{\Omega'\setminus A_k} \abs{Dv_j(k)}^2 - \int_{\Omega'} \abs{Db_j(k)} \ge 0.
	\]
	\item[(b)] $\liminf_{k \to \infty} \int_{\Omega} \left(\abs{Dv_j(k)}^2 -\abs{Db_j(k)}^2\right) = 0$ if and only if\\ ${\liminf_{k \to \infty} \int_{\Omega} \bigl(\abs{Dv_j(k)} - \abs{Db_j(k)})^2 = 0}$.
\end{itemize}
Due to properties (1) to (4) we may sum in $j$ and replace $\sum_{j=1}^J v_j(k)$ by $u(k)$. This completes the proof. 
\end{proof}

\section{The Luckhaus lemma}\label{sec:ALuckhaus}
A classical result due to S.~Luckhaus is concerned with the extension of a map that is defined on the boundary of an annulus $\partial \left( B_1\setminus B_{1-\lambda} \right)$ into the interior. Its proof for single valued functions is nowadays classical and can be found for instance in \cite{Moser}. We mentioned the result already in section \ref{sec:some_classical_results_extended_to_q_valued_functions}. We want to give now a complete intrinsic proof for $Q$-valued functions.
Our formulation is based on S.~Luckhaus' original, \cite[Lemma 1]{Luckhaus} and the one of R.~Mosers, \cite[Lemma 4.4]{Moser}.

\begin{lemma}\label{lem:Luckhaus}
There is a constants $C, C_\infty>0$ depending only on the dimension $N$ such that the following holds:\\
Suppose $\lambda=\frac{1}{L}$, $\epsilon=\frac{1}{lL}\le \lambda$, $l,L \in \N$,$L>2$ given, furthermore let $u, v \in W^{1,2}(\Sp^{N-1}, \A_Q(\R^m))$ with
\begin{equation}\label{eq:A001}
\int_{\Sp^{N-1}} \abs{D_\tau u}^2 + \abs{D_\tau v}^2 + \frac{\G(u,v)^2}{\epsilon^2} = K^2;
\end{equation}
then there exists $\varphi \in W^{1,2}(B_1\setminus B_{1-\lambda}, \A_Q(\R^m))$ with the following properties
\begin{align}
\varphi(x)&= \begin{cases} u(x), &\text{ if } \abs{x}=1\\ v(\frac{x}{1-\lambda}), &\text{ if } \abs{x}=1-\lambda \end{cases} \label{eq:A002}\\
\int_{B_1\setminus B_{1-\lambda}} \abs{D\varphi}^2 &\le C\, Q\, \lambda K^2 \label{eq:A003}\\
\varphi(x) &\in \{ y \in \R^m \colon \dist(y, u(\Sp^{N-1}) \cup v(\Sp^{N-1}) )< a\} \label{eq:A004}\\
&\text{ for some $a>0$ with } a^2 \le C_\infty \, Q^2\, \lambda^{2-N} \epsilon K^2 \nonumber.
\end{align}
\end{lemma}
\begin{remark}\label{rem:A001}
The $L^\infty$-bound, \eqref{eq:A004}, is a little bit weaker then tho stated in Lemma \ref{lem:3.101}. The dependence of the constants on $N, m$ and $Q$ is more precise. 
\end{remark}

The proof of Lemma \ref{lem:Luckhaus} is very close to S.~Luckhaus orginial one, nicely presented by R.~Moser, \cite[Lemma 4.4]{Moser}. It splits in 3 parts:
\begin{enumerate}
\item a decomposition $\G$ of the sphere $\Sp^{N-1}$ that is bilipschitz to cubical decomposition of $\partial [-1,1]^N$ into parallel disjoint cubes of side length $\lambda$. This is a measure theoretic argument;
\item two types of extensions on cubes;
\item a recursive definition of $\phi$ on cubical subsets $F\times [0,\lambda]$ where $F$ is a $k$-dimensional face int the cubical decomposition. It always takes advantage that $\phi$ had already be defind on all $F'\times [0,\lambda]$ for lower dimensional faces $F'$. 
\end{enumerate}

Studying S.~Luckhaus' original proof one notice that only for the extensions on $F\times [0,\lambda]$, $F$ being a $1$-dimensional face, the linear structure of $W^{1,2}(F, \R^m)$ is needed. C.~De Lellis presented a possible replacement $W^{1,2}(F, \A_Q(\R^m))$ in \cite{Lellis err}. His version does not preserve the $L^\infty$-bound, \eqref{eq:A004}, compare remark below. Nonetheless following his ideas one can recover the bound, lemma. Our proof does not contain essentially new ideas. It boils down to replacing lemma E.2 in the proof proposed by C.~De Lellis, \cite{Lellis err} or the linear extension in S.~Luckhaus original one by lemma. Nonetheless we decided to give a complete detailed proof.\\

As mentioned, in part 1 one uses the bilipschitz equivalence between $B_1$ and $[-1,1]^N$ and their boundaries $\Sp^{N-1}$ and $\partial [-1,1]^N$.  Therefore we list in the following remark some terminology and constants appearing in this context. Since only the extensions, part 2, differ slightly from the already existing proofs they are presented first thereafter. Finally we will proceed with part 1 and 3.

\begin{remark}\label{rem:A002}
$\abs{x}^2=\abs{x}^2_2= (x_1)^2+ \dotsb + (x_n)^2$ denotes the Euclidean norm on $\R^n$ and $\abs{x}_\infty= \max\{ \abs{x_1}, \dotsc, \abs{x_n}^2 \}$ the supremum norm. Let $B_1=\{ \abs{x}_2 < 1 \}$ be the unite ball and $[-1,1]^n=\{ \abs{x}_\infty < 1\}$ the standard cube in $\R^n$. Set $H(x)= \frac{\abs{x}_\infty}{\abs{x}_2} x$ and $G(x)= \frac{\abs{x}_2}{\abs{x}_\infty} x$, then $H=G^{-1}$ and $H: [-1,1]^n \to B_1$ so their boundaries $H: \partial [-1,1]^n \to \Sp^n$. $\delta_{ij}$ denote the Euclidean metric on $\R^n$ or the pullback metric for a submanifold in $\R^n$. Furthermore let $g=G^\sharp\delta$ and $h=H^\sharp\delta$ be the pullback metrics on $B_1, [-1,1]^n$ respectively. One calculates 
\[\det(g)=\left( \frac{\abs{x}_2}{\abs{x}_\infty} \right)^{2n} = \det(g\tr{\Sp^{n-1}}). \]
Furthermore the spectrum of $g^{-1}$ is contained in $[1-\left(\frac{\abs{x}_\infty}{\abs{x}_2}\right)^2, 1+ \left(\frac{\abs{x}_\infty}{\abs{x}_2}\right)^2 ]$. The eigenvalues of $g\tr{\Sp^{n-1}}$ are $\left(\frac{\abs{x}_2}{\abs{x}_\infty}\right)^4$ and $n-2$ times $\left(\frac{\abs{x}_2}{\abs{x}_\infty}\right)^2$. For $h$ we therefore have
\[\det(h)=\left( \frac{\abs{x}_\infty}{\abs{x}_2} \right)^{2n} = \det(h\tr{\partial [-1,1]^n}). \]
The spectrum of $h^{-1}$ is contained in $[\left(\frac{\abs{x}_2}{\abs{x}_\infty}\right)^4-\left(\frac{\abs{x}_2}{\abs{x}_\infty}\right)^2, \left(\frac{\abs{x}_2}{\abs{x}_\infty}\right)^4+ \left(\frac{\abs{x}_2}{\abs{x}_\infty}\right)^2 ]$. The eigenvalues of $h\tr{\partial [-1,1]^n}$ are $\left(\frac{\abs{x}_\infty}{\abs{x}_2}\right)^4$ and $n-2$ times $\left(\frac{\abs{x}_\infty}{\abs{x}_2}\right)^2$.
This has for instance the following implications:
\begin{align}\label{eq:A005}
\int_{B_1} \abs{D\varphi}^2 &= \int_{[-1,1]^n} h^{ij} \frac{\partial \phi}{\partial x^i}\frac{\partial \phi}{\partial x^j} \sqrt{\det(h)} \le 3 \int_{[-1,1]^n} \abs{D\phi}^2 \text{ for } \varphi=\phi\circ G \\ \nonumber
\int_{\Sp^{n-1}} \abs{D_\tau \varphi}^2 &= \int_{\partial[-1,1]^n} h\tr{\partial[-1,1]^n}^{ij} \frac{\partial \phi}{\partial x^i}\frac{\partial \phi}{\partial x^j} \sqrt{\det(h\tr{\partial[-1,1]^n})} \le c_n \int_{\partial [-1,1]^n} \abs{D_\tau\phi}^2,
\end{align}
since $\left\{ \left(\frac{\abs{x}_2}{\abs{x}_\infty}\right)^4+ \left(\frac{\abs{x}_2}{\abs{x}_\infty}\right)^2 \right\} \left( \frac{\abs{x}_\infty}{\abs{x}_2} \right)^{n} \le 3 \; \forall n$ and $\left(\frac{\abs{x}_2}{\abs{x}_\infty}\right)^{4-n}\le c_n$ for $c_2=2,c_3=\sqrt{3}$ and $c_n=1$ for $n\ge 4$. Similarly one calculates for $ \phi=\varphi\circ H $
\begin{align}\label{eq:A006}
\int_{[-1,1]^n} \abs{D\phi}^2 &= \int_{B_1} g^{ij} \frac{\partial \varphi}{\partial x^i}\frac{\partial \varphi}{\partial x^j} \sqrt{\det(g)} \le n^{\frac{n}{2}}(1+n^{-1}) \int_{B_1} \abs{D\varphi}^2\\ \nonumber
\int_{\partial [-1,1]^n} \abs{D_\tau\phi}^2 &= \int_{\Sp^{n-1}} g\tr{\Sp^{n-1}}^{ij} \frac{\partial \varphi}{\partial x^i}\frac{\partial \varphi}{\partial x^j} \sqrt{\det(g\tr{\Sp^{n-1}})} \le n^{\frac{n-2}{2}} \int_{\Sp^{n-1}} \abs{D_\tau\phi}^2.
\end{align} 
\end{remark}

The extension lemma for faces of dimension $k\ge3$ is the classical following one: 
\begin{lemma}\label{lem:A002}
Given $F=z+[0, \lambda]^n$, $n\ge 3$, a $n$-dimensional cube of side length $\lambda$ and $\phi\in W^{1,2}(\partial F, \A_Q(\R^m))$ then there is an extension $\widehat{\phi}\in W^{1,2}(F,\A_Q(\R^m))$ with the property that
\begin{align}
\int_{F} \abs{D\widehat{\phi}}^2 \le \frac{n}{2(n-2)} \, \lambda \int_{\partial F} \abs{D_\tau \phi}^2 \label{eq:A007}\\
\widehat{\phi}(x) \in \phi(\partial F) \quad \forall x \in F \label{eq:A008}
\end{align}
\end{lemma}
\begin{proof}
By a simple scaling argument it is sufficient to prove the lemma for $F=[-1,1]^n$. Since $n\ge 3$ the $0$-homogeneous extension $\widehat{\phi}(x) = \phi\left(\frac{x}{\abs{x}_\infty}\right)$ belongs to $W^{1,2}(F,\A_Q(\R^m))$. Direct computations provide the bound \eqref{eq:A005}. \eqref{eq:A006} is clearly satisfied.
\end{proof}
The crucial point is to find a "version" of Lemma \ref{lem:A002} for $n=2$. The first step is the replacement suggested by C.~De Lellis.
\begin{lemma}\label{lem:A003}
Given $F=z+[0, \lambda]^2$, a $2$-dimensional cube of side length $\lambda$ and $\phi\in W^{1,2}(\partial F, \A_Q(\R^m))$ then there is an extension $\widehat{\phi}\in W^{1,2}(F,\A_Q(\R^m))$ with the property that
\begin{align}
\int_{F} \abs{D\widehat{\phi}}^2 \le 3Q\, \lambda \int_{\partial F} \abs{D_\tau \phi}^2 \label{eq:A009}\\
\G(\widehat{\phi}(x),\widehat{\phi}(y))^2 \le \pi Q^2\, \lambda \int_{\partial F} \abs{D_\tau \phi}^2.\label{eq:A010}
\end{align}
\end{lemma}
\begin{proof}
By scaling it is sufficient to prove it for $F=[-1,1]^2$. Furthermore using $\varphi=\phi\circ G, \widehat{\phi}=\widehat{\varphi}\circ H$ and the estimates \eqref{eq:A005}, \eqref{eq:A006} for $n=2$ we can show the existence of an extension $\widehat{\varphi}$ from $\Sp^1$ to the disk $B_1$, that satisfies
\begin{align}
\int_{B_1} \abs{D\widehat{\varphi}}^2 \le Q\, \int_{\Sp^1} \abs{D_\tau \varphi}^2 \label{eq:A011}\\
\G(\widehat{\varphi}(x),\widehat{\varphi}(y))^2 \le \pi Q^2\, \int_{\Sp^1} \abs{D_\tau \varphi}^2.\label{eq:A012}
\end{align}
The energy bound \eqref{eq:A011} is derived in Proposition 3.10 in \cite{Lellis} as the crucial estimate to establish the optimal H\"older continuity for Dirichlet minimizers in the interior. Although the competitor constructed there satisfies the $L^\infty$-bound it is not stated. Therefore we present the complete construction. Recall that for a given $\tilde{f} \in W^{1,2}(\Sp^{1}, \R^m)$, single valued, there exists a unique harmonic extension $f \in W^{1,2}(B_1, \R^m)$ ( $\Delta f =0$) with $f= \tilde{f}$ on $\Sp^1$ and it satisfies
\begin{equation}\label{eq:A013}
\int_{B_1} \abs{Df}^2 \le \int_{\Sp^1} \abs{D_\tau \tilde{f}}^2
\end{equation}
and due to the maximum principle for subharmonic functions and $1$-dimensional calculus 
\begin{equation}\label{eq:A014}
\abs{f(x) - f(y)}^2 \le \sup_{x,y \in \Sp^1} \abs{\tilde{f}(x)-\tilde{f}(y)}^2 \le \pi \int_{\Sp^1} \abs{D_\tau \tilde{f}}^2.
\end{equation}
Now let be $\varphi \in W^{1,2}(\Sp^1, \A_Q(\R^m))$ given, as shown in \cite[Proposition 1.5]{Lellis} there is an irreducible decomposition $\varphi(x) = \sum_{j=1}^J \sum_{\substack{z \in \C\\ z^{Q_j} =x }} \llbracket \tilde{g}_j(z) \rrbracket$ for all $x \in \Sp^1$, functions $\tilde{g}_j \in W^{1,2}(\Sp^1, \R^m)$ and $\sum_{j=1}^J Q_j =Q$. To every $\tilde{g}_j$ let $g_j \in W^{1,2}(B_1, \R^m)$ be the harmonic extension, then set 
\[ \widehat{\varphi}(x) = \sum_{j=1}^J \sum_{\substack{z \in \C\\ z^{Q_j} =x }} \llbracket g_j(z) \rrbracket \text{ for } x \in B_1.\]
Direct computations, compare \cite[Lemma 3.12]{Lellis} and \eqref{eq:A013} gives \eqref{eq:A011}:
\begin{align*}
\int_{B_1} \abs{D\widehat{\varphi}}^2 &= \int_{B_1} \sum_{j=1}^J \abs{Dg_j}^2 \le \sum_{j=1}^J \int_{\Sp^1} \abs{D_\tau \tilde{g}_j}^2\\ &\le Q  \sum_{j=1}^J \frac{1}{Q_j} \int_{\Sp^1} \abs{D_\tau \tilde{g}_j}^2 = Q \int_{\Sp^1}\abs{D_\tau \varphi}^2.  
\end{align*}
Furthermore let $x=r\exp(i\alpha)$ then \[\widehat{\varphi}(x)= \sum_{j=1}^J \sum_{l=0}^{Q_j-1} \llbracket g_j\left( r^{\frac{1}{Q_j}} \, e^{i \frac{\alpha}{Q_j} +il\frac{2\pi}{Q_j}} \right)\rrbracket\]
similar for $y=s\exp(i\beta)$, hence applying \eqref{eq:A014} gives \eqref{eq:A012}
\begin{align*}
\G(\widehat{\varphi}(x), \widehat{\varphi}(y))^2 &\le \sum_{j=1}^J \sum_{l=0}^{Q_j-1} \Abs{ g_j\left( r^{\frac{1}{Q_j}} \, e^{i \frac{\alpha}{Q_j} +il\frac{2\pi}{Q_j}} \right) - g_j\left( s^{\frac{1}{Q_j}} \, e^{i \frac{\beta}{Q_j} +il\frac{2\pi}{Q_j}} \right)}^2 \\
&\le \pi \sum_{j=1}^J \sum_{l=0}^{Q_j-1} \int_{\Sp^1} \abs{D_\tau \tilde{g}_j}^2 \le \pi Q^2\, \int_{\Sp^1} \abs{D_\tau \varphi}^2.
\end{align*}
\end{proof}

Although $\A_Q(\R^m)$ is not a linear space we will use the following terminology.  A map $\phi: [a,b] \to \A_Q(\R^m)$ is said to be linear, a linear interpolation, between two points $S=\sum_{l=1}^Q \llbracket s_l \rrbracket, T=\sum_{l=1}^Q \llbracket t_l \rrbracket \in \A_Q(\R^m)$ on the interval $[a,b]$ if there exists $\sigma \in \mathcal{P}_Q$ such that
\begin{align*}
\G(S,T)^2 &= \sum_{l=1}^Q \abs{s_l - t_{\sigma(l)}}^2\\
\phi(t) &= \sum_{l=1}^Q \llbracket \frac{b-t}{b-a} s_l + \frac{t-a}{b-a} t_{\sigma(l)} \rrbracket.
\end{align*}
Furthermore one has $\int_a^b \abs{D\phi}^2 = \frac{\G(S,T)^2}{b-a}$ and to any two points $S,T \in \A_Q(\R^m)$ and an interval $[a,b]$ given there exists at least one linear interpolation. (It may not be unique.)

\begin{lemma}\label{lem:A004}
Suppose $\phi \in  W^{1,2}(\partial (F\times [0,\lambda]), \A_Q(\R^m))$, $F=[a,b]$ a $1$-dimensional face of length $\lambda=b-a$  is given and $\epsilon= \frac{\lambda}{l}, l \in \N$. Furthermore $\phi$ satisfies the following:
\begin{align*}
&t\mapsto \phi(a,t), \phi(b,t) \text{ are linear between $U(a),V(a)$ and $U(b),V(b)$};\\
&\int_{F}  \abs{D_\tau U}^2 + \abs{D_\tau V}^2 + \frac{\G(U,V)^2}{\epsilon^2} = K^2.
\end{align*}
where $U(x)=\phi(x,0), V(x)=\phi(x,\lambda) \in W^{1,2}(F, \A_Q(\R^n)$
Then there exists an extension $\widehat{\phi} \in W^{1,2}(F\times[0,\lambda],\A_Q(\R^m))$ satisfying 
\begin{align}\label{eq:A015}
&\int_{F\times [0,\lambda]} \abs{D\widehat{\phi}}^2 \le 15Q\, \lambda K^2;\\ \label{eq:A016}
&\dist(\widehat{\phi(x,t)}, U(F)\cup V(F))^2 \le 5 \pi Q^2\, \epsilon K^2 \quad \forall (x,t) \in F\times[0,\lambda].
\end{align}
\end{lemma}
\begin{proof}
We construct $\widehat{\phi}$ applying the previous extension lemma \ref{lem:A003} several times. Set $a_k=a+ k \epsilon$ for $k=0, \dotsc, l$, i.e. $a_0=a, a_l=b$ and for every $k=1, \dotsc, l-1$ define $t \mapsto \phi(a_k, t)$ to be a linear interpolation between $U(a_k), V(a_k)$. \\
Pick any $k\in \{0, \dotsc, l-1 \}$ then $\phi$ is now already defined on $\partial ([a_{k}, a_{k+1}]\times [0, \lambda])$. We may apply lemma \ref{lem:A003} to 
\[ (x,t) \in [0,\lambda]^2 \mapsto \phi(a_{k} + \frac{x}{l}, t)\]
and obtain an extension $\phi_k \in W^{1,2}([0,\lambda]^2, \A_Q(\R^m))$. By $1$-dimensional calculus one has for $f \in W^{1,2}([c,d], \R)\subset C^{0,\frac{1}{2}}([c,d])$, that
\[\sup_{c\le x \le d} f(x)^2 \le 2\abs{d-c} \int_c^d \abs{f'}^2 + \frac{2}{\abs{d-c}} \int_c^d f^2 \]
and therefore
\[ \sum_{j=k}^{k+1}\G(U(a_j),V(a_j))^2 \le 4\epsilon \int_{a_k}^{a_{k+1}} \abs{D_\tau U}^2+ \abs{D_\tau V}^2 + \frac{4}{\epsilon} \int_{a_k}^{a_{k+1}} \G(U,V)^2 = 4 \epsilon K^2\]. 
This gives ($\frac{\epsilon}{\lambda}= \frac{1}{l}$)
\begin{align*}
\int_{\partial [0,\lambda]^2} \abs{D_\tau \phi_k}^2 &= \frac{1}{l} \left(\int_{a_k}^{a_{k+1}} \abs{D_\tau U}^2 + \abs{D_\tau V}^2 \right)+ \sum_{j=k}^{k+1}\frac{\G(U(a_j),V(a_j))^2}{\lambda}\\
&\le \frac{5}{l} K^2. 
\end{align*}
Finally we define
\[ \widehat{\phi}(x,t) = \phi_k(l(x-a_k),t) \text{ for } (x,t) \in [a_k, a_{k+1}] \times [0,\lambda]. \]
Due to lemma \ref{lem:A003} we found
\begin{align}\label{eq:A017}
\int_{[a_k,a_{k+1}]\times [0,\lambda]} \abs{D\widehat{\phi}}^2 \le l \int_{[0,\lambda]^2} \abs{D\phi_k}^2 \le 3Q\, l\lambda \int_{\partial[0,\lambda]^2 } \abs{D_\tau \phi_k}^2 \le 15Q\, \lambda K^2 \\ \nonumber
\G(\widehat{\phi}(x,t),U(x))^2= \G(\phi_k(y,t),\phi_k(y,0))^2 \le 5\pi Q^2\, \epsilon K^2\quad \forall x=a_k+\frac{y}{l}, y \in [0,\lambda]. 
\end{align}
Since all sets $[a_k, a_{k+1}[\times [0,\lambda]$ are disjoint we obtain a well defined extension $\widehat{\phi}$ applying the above procedure for every $k=0, \dotsc, l-1$. Furthermore adding the estimate \eqref{eq:A017} for $k=0, \dotsc, l-1$ we obtain \eqref{eq:A016} proving the lemma.
\end{proof}

The choice $l=1$ in lemma \ref{lem:A004} reduces it back to lemma \ref{lem:A003}. This corresponds to C.~De Lellis proposal in \cite{Lellis err} to choose the "harmonic" extension. This is in general not a good idea for the $L^\infty$-bound. This can be seen in the following example.

\begin{example}\label{ex:A003}
Let $F=[0,1], (\lambda=1)$, $M \in \N$ and $\phi_M(x,0)=\phi(x,\lambda) = M \Bigl( \cos(2\pi M x), \sin(2\pi M x)\Bigr) \in W^{1,2}(F, \R^2)$, $\phi(0,t)=\phi(\lambda,t)\equiv\Bigl(0,1\Bigr)$. S.~Luckhaus suggests the extension $\widehat{\phi}_L(x,t)= \phi(x,0)$ for all $t \in [0,\lambda]$ that satisfies $\dist(\widehat{\phi}_L(x,t), \phi(F,0))^2=0$ for all $(x,t)\in F\times [0,\lambda]$.
The harmonic extension would be \[\widehat{\phi}_H(x,t) = \frac{\cosh(2\pi M (t-\frac{1}{2}))}{\cosh(\pi M)} \phi(x,0);\] that satisfies now
\[\inf_{x \in F} \abs{\phi_H(x,\frac{1}{2})- \phi(x,0)} \ge \abs{\phi(x,0)}- \abs{\phi_H(x, \frac{1}{2}} = M \left( 1- \frac{1}{\cosh(\pi M)}\right);\]
converging to $+\infty$ as $M \to \infty$.
\end{example}

\begin{proof}[Proof of Lemma \ref{lem:Luckhaus} ]
(Our presentation is close to the proof presented by R.~Moser in \cite{Moser}.)\\

\emph{Part 1: decomposition $\G$ of the sphere using a Fubini-type argument}\\
It is useful to set up some terminology. $\frac{1}{L}\Z^N$ is a square lattice in $\R^N$ decomposing the cube $[-1,1]^N$ and ist boundary $\partial [-1,1]^N$ into congruent cubes of side length $\frac{1}{L}$ of dimension $N$ and $N-1$.  Let $\F_k$ denote the collection of all $k$-dimensional faces in the decomposition $\partial [-1,1]^N \cap \frac{1}{L} \Z^N$. We set $\G_k=\{ H(F) \colon F \in \F_k \}$, a collection of $k$-dimensional faces on the sphere $\Sp^{N-1}$. The number of $k$-dimensional faces $\sharp F_k = \sharp G_k$ is less then $2N$-times the number of $k$-dimensional faces in $[-1,1]^{N-1} \cap \frac{1}{L} \Z^{N-1}$, that is less than $(2L)^{N-1} \tbinom{N-1}{k}$, in total
\begin{equation}\label{eq:A018}
\sharp \F_k = \sharp G_k \le N 2^{N} L^{N-1} \tbinom{N-1}{k}.
\end{equation}
\emph{claim: } Let $f\in L^1(\Sp^{N-1}, \R_+)$ be given. Then there is a partition of $SO(N)$ into the set $\mathcal{O}^\text{good}$ of "good" and the set $\mathcal{O}^\text{bad}$ of "bad" matrices, defined as follows: $O\in \mathcal{O}^\text{good}$ if we have 
\begin{equation}\label{eq:A019}
\sum_{k=1}^{N-2} \frac{L^{N-1-k}}{\tbinom{N-1}{k}} \sum_{G \in \G_k} \int_G f(Ox) \,d\h^k(x) \le \frac{(N-2)2^{N}}{\theta w_N} \int_{\Sp^{N-1}} f \, d\h^{N-1}
\end{equation}
and $O \in \mathcal{O}^\text{bad}$ if instead 
\begin{equation}\label{eq:A019b}
\sum_{k=1}^{N-2} \frac{L^{N-1-k}}{\tbinom{N-1}{k}} \sum_{G \in \G_k} \int_G f(Ox) \,d\h^k(x) > \frac{(N-2)2^{N}}{\theta w_N} \int_{\Sp^{N-1}} f \, d\h^{N-1}.
\end{equation}
Furthermore one has $\mu(\mathcal{O}^\text{bad})< \theta$, where $\mu$ is the Haar measure on $SO(N)$.\\

This can be seen as follows: To any $x,x_0 \in \Sp^{N-1}$ there exists $O_0 \in SO(N)$ with $O_0x_0=x$ and by the invariance of the Haar measure under group action we have
\[ \int_{O \in SO(N)} f(Ox) \,d\mu(O) = \int_{SO(N)} f(OO_0x_0) \,d\mu(O) = \int_{SO(N)} f(Ox_0) \,d\mu(O).\]
The invariance of the Haussdorff measure under orthogonal transformations gives
\[ \int_{\Sp^{N-1}} f(O x) \,d\h^{N-1}(x) = \int_{\Sp^{N-1}} f(x) \,d\h^{N-1}(x). \]
Fubini's theorem with $\mu(SO(N))=1$ gives
\begin{align*} \int_{\Sp^{N-1}} f \,d\h^{N-1} &= \int_{\Sp^{N-1}} f(Ox) \,d\h^{N-1}(x) = \int_{SO(N)} \int_{\Sp^{N-1}} f(Ox) \,d\h^{N-1}(x) \,d\mu(O)\\ &= Nw_N \int_{SO(N)} f(Ox_0) \,d\mu(O).\end{align*}
We deduce
\begin{align}\nonumber
&\int_{SO(N)} \sum_{G \in \G_k} \int_G f(Ox) \,d\h^{k}(x) \,d\mu(O) = \sum_{G \in \G_k} \int_{SO(N)} f(Ox_0) \,d\mu(O) \h^k(G)\\ \label{eq:A020} 
&\le 2^N N \tbinom{N-1}{k} L^{N-1-k} \int_{SO(N)} f(Ox_0) \,d\mu(O) = \frac{2^N}{w_N} \tbinom{N-1}{k} L^{N-1-k} \int_{\Sp^{N-1}} f \,d\h^{N-1}. 
\end{align}
We used \eqref{eq:A018} and $\h^k(G)=\h^k(H(F)) \le \h^k(F) = L^{-k}$. This implies the claim $\mu(\mathcal{O}^\text{bad})< \theta$ because apply \eqref{eq:A020} for every $k$ and \eqref{eq:A019b} for every $O \in \mathcal{O}^\text{bad}$ to deduce
\begin{align*}
&\frac{\mu(\mathcal{O})}{\theta} \frac{(N-2)2^{N}}{w_N} \int_{\Sp^{N-1}} f \, d\h^{N-1} \\&< \int_{O \in \mathcal{O}} \sum_{k=1}^{N-2} \frac{L^{N-1-k}}{\tbinom{N-1}{k}} \sum_{G \in \G_k} \int_G f(Ox) \,d\h^k(x) \,d\mu(O) \le \frac{(N-2)2^{N}}{w_N} \int_{\Sp^{N-1}} f \, d\h^{N-1};
\end{align*}
i.e. $\mu(\mathcal{O}^\text{bad})< \theta$.\\

Given $u,v$ as assumed, set $\theta= \frac{1}{2}$ and $f_1= \abs{Du}^2 + \abs{Dv}^2 + \frac{\G(u,v)^2}{\epsilon^2}$, $f_2= \abs{u}^2 + \abs{v}^2$. The the claim states that if $\mathcal{O}_i^\text{good} \cup \mathcal{O}^\text{bad}_i=SO(N)$ are the related partition, there exists $O \in \mathcal{O}^\text{good}_1 \cap \mathcal{O}^\text{good}_2$ since $\mu(\mathcal{O}^\text{bad}_1 \cup \mathcal{O}^\text{bad}_2)<1$. Hence we have for any $k=1, \dotsc , N-2$, $G \in \G_k$
\begin{align*}
&u\circ O \tr{G}, \;v\circ O \tr{G} \in W^{1,2}(G, \A_Q(\R^m))\\
&(u\circ O \tr{G})\tr{G'}= u\circ O\tr{G'} ,\; (v\circ O \tr{G})\tr{G'}= v\circ O\tr{G'} \quad \forall G' \in \G_{k-1}, G' \subset \partial G.
\end{align*}
We define $U(x)=u(OH(x)), V(x)=v(OH(x))$. Due to the choice of $O$ we have that for any $k=1, \dotsc , N-2$, $F \in \F_k$
\begin{align*}
&U\tr{F},\; V\tr{F} \in W^{1,2}(F, \A_Q(\R^m))\\
&(U\tr{F})\tr{F'}= U\tr{F'}, \;(V\tr{F})\tr{F'}= V\tr{F'} \quad \forall F' \in \F_{k-1}, F' \subset \partial F.
\end{align*}
Set $\tilde{f}_1= \abs{DU}^2 + \abs{DV}^2 + \frac{\G(U,V)^2}{\epsilon^2}$ and using remark \ref{rem:A002} we have for any $F \in \F_k$
\[ \int_F \tilde{f}_1 \,d\h^k \le \int_{G=H(F)} \left(\frac{\abs{x}_2}{\abs{x}_\infty}\right)^{k-1} f_1(Ox) \,d\h^{k}(x) \le N^{\frac{k-1}{2}} \int_G f_1(Ox) \,d\h^k(x). \]
so that 
\begin{equation}\label{eq:A021}
\sum_{k=1}^{N-2} \frac{L^{N-1-k}}{N^{\frac{k-1}{2}} \tbinom{N-1}{k}} \sum_{F \in \F_k} \int_F \tilde{f}_1 \,d\h^k \le \frac{(N-2)2^{N+1}}{w_N} \; K^2.
\end{equation}

\emph{Part 2: extensions of maps that are defined on the boundary of a $k$-dimensional cube $\partial F$ to its interior}\\
This is covered in the results of lemma \ref{lem:A002} and \ref{lem:A004}.\\

\emph{Part 3: recursive construction of $\phi$}\\
We define $\phi$ on $F\times [0,\lambda]\; \forall F \in \F_1$ using lemma \ref{lem:A004}, then recursively on $\{  F \times [0,\lambda] \colon F \in \F_2 \}, \{  F \times [0,\lambda] \colon F \in \F_3 \}, \dotsc, \{  F \times [0,\lambda] \colon F \in \F_{N-1} \}$ by lemma \ref{lem:A002}. In each step taking advantage of the fact that $\phi$ had already be defined on the boundary of $F \times [0,\lambda]$, with 
\begin{equation}\label{eq:A022}
\phi(x,0) = U(x), \; \phi(x,\lambda)= V(x)\quad \forall x \in F, F\in \F_k.
\end{equation}
Now we describe the construction in detail. ($D_\tau$ denotes the tangential differential with respect to the domain of integration, i.e. $\abs{D_\tau \phi}^2$ will be the Dirichlet energy with respect to $F\times [0,\lambda]$, $\abs{D_\tau U}^2 + \abs{D_\tau V}^2$ the Dirichlet energy with respect to a face $F$.):
Pick $z\in \F_0$, the set of all vertices,  define
\begin{equation}\label{eq:A023}
\phi(z,t)\in W^{1,2}(\{z\}\times [0,\lambda], \A_Q(\R^m))\; t\mapsto\phi(z,t) \text{ linear between } U(z), V(z).
\end{equation}
We proceed this way for all $z\in \F_0$: since $z \in \partial F'$ for some $F' \in \F_1$ and $W^{1,2}(F',\A_Q(\R^m))\subset C^{0,\frac{1}{2}}(F',\A_Q(\R^m))$ , $U(z), V(z)$ are defined. Furthermore all $\{z\}\times[0,\lambda]$ are disjoint so $\phi$ is welldefined on $\bigcup_{z\in \F_0} \{z\}\times[0,\lambda]$.

Pick $F \in \F_1$ then $\phi$ is already defined on $\partial\left(F\times [0,\lambda]\right)=F\times\{0,\lambda\} \cup \partial F \times [0,\lambda]$ taking into account \eqref{eq:A022} and \eqref{eq:A023}. We apply lemma \ref{lem:A002} to extend $\phi$ to $F\times[0,\lambda]$ with the estimates: $ \int_{F\times[0,\lambda]} \abs{D_\tau \phi}^2 \le 15 Q\, \lambda K^2_F$, $\dist(\phi(x,t), U(F)\cup V(F))^2 \le 15 Q^2\, \epsilon  K_F^2$  with $K_F^2= \left(\int_{F} \abs{D_\tau U}^2 +\abs{D_\tau V}^2 + \frac{\G(U,V)^2}{\epsilon^2} \, d\h^1 \right)$.
We can define $\phi$ for all $F \in \F_1$ since the interior of the sets $F\times[0,\lambda]$ are disjoint. Taking into account \eqref{eq:A021} we found (with $C_1\le \frac{2^{N+5}(N-1)^2}{w_N}$)
\begin{align}\label{eq:A024}
&\sum_{F \in \F_1} \int_{F\times [0,\lambda]} \abs{D_\tau \phi}^2 \le 15 Q\, \lambda \sum_{F \in \F_1} K_F^2 \le C_1Q\, \lambda^{3-N} K^2 \\ \label{eq:A025}
&\dist(\phi(x,t), U(F)\cup V(F))^2 \le C_1 Q^2 \,\epsilon \lambda^{2-N} K^2 \quad (x,t) \in \bigcup_{F\in \F_1} F\times[0,\lambda].
\end{align}
Pick $F\in \F_2$, then $\phi$ is defined on $\partial\left(F\times [0,\lambda]\right)=F\times\{0,\lambda\} \cup \partial F \times [0,\lambda]$, taking into account \eqref{eq:A022} and the previous step ($\partial F= \bigcup_{i=1}^4 F_i, F_i \in \F_1$). Hence $\phi$ can be extended to $F\times[0,\lambda]$ using lemma \ref{lem:A002} s.t. $\phi(x,t) \in \{ \phi(y,s) \colon (y,s) \in \partial (F\times[0,\lambda])\}$ and
\[\int_{F\times[0,\lambda]} \abs{D_\tau \phi}^2 \le \frac{3}{2}\lambda \left( \int_{F} \abs{D_\tau U}^2 + \abs{D_\tau V}^2 \,d\h^2 + \sum_{i=1}^4 \int_{F_i\times[0,\lambda]} \abs{D_\tau \phi}^2 \right).\] 
As before the interior of the sets $F\times[0,\lambda]$, $F\in\F_2$ are disjoint, so we can proceed this way for all of them and obtain a welldefined $\phi$ on $\bigcup_{F\in \F_2} F\times[0,\lambda]$. Summing the above estimate for all $F\in \F_2$, taking into account \eqref{eq:A021} and \eqref{eq:A024} we get for some constant $C_2$: 
\begin{equation*}
\sum_{F \in \F_2} \int_{F \times [0,\lambda]} \abs{D_\tau \phi}^2 
\le C_2Q\, \lambda^{4-N} K^2.
\end{equation*}
(For a given $F\in \F_k$ we have $\sharp \{ F'\in \F_{k+1} \colon F \subset \partial F \} \le 2 (N-1-k)$.)\\
We use the same method to define $\phi$ on $\{  F \times [0,\lambda] \colon F \in \F_3 \}, \dotsc, \{  F \times [0,\lambda] \colon F \in \F_{N-1} \}$. Each time we obtain the inequality 
\[ \sum_{F\in \F_k} \int_{F\times[0,\lambda]} \abs{D_\tau \phi}^2 \le C_kQ\, \lambda^{k+2-N} K^2. \]
For $k=N-1$ this is 
\begin{equation}\label{eq:A026} \int_{\partial[-1,1]^N \times [0,\lambda]} \abs{D_\tau \phi}^2 \le C_{N-1}Q\, \lambda K^2. \end{equation}
Applying lemma \ref{lem:A002} does not affect the $L^\infty$ bound, \eqref{eq:A025}.\\
Define $\varphi(x)=\varphi(r y) = \phi(G \circ O^t(y), 1-r) \in W^{1,2}(B_1\setminus B_{1-\lambda}, \A_Q(\R^m))$, with $r=\abs{x}, y=\frac{x}{\abs{x}}$. One checks that $\phi$ satisfies \eqref{eq:A002}. \eqref{eq:A026} combined with remark \ref{rem:A002} gives the energy bound \eqref{eq:A003}:
\[ \int_{B_1\setminus B_{1-\lambda}} \abs{D\varphi}^2 \le 4 \int_{\partial[-1,1]^N \times [0,\lambda]} \abs{D\phi}^2 \le CQ\, \lambda K^2.\]
Finally the preserved $L^\infty$ bound \eqref{eq:A025} corresponds with \eqref{eq:A004}.
\end{proof}
\end{appendix}

\end{document}